\documentclass[a4paper,12pt]{article}
\usepackage{amssymb,amsmath,amsthm}
\usepackage{bm}%
\usepackage{stmaryrd}
\usepackage[pdftex]{hyperref}
\hypersetup{
	colorlinks=true, 
	citecolor=blue,
	linkcolor=blue,
	urlcolor=orange,
}
\usepackage{tabularray}
\usepackage{tikz,tikz-cd}
  \usetikzlibrary{shapes,snakes,calc,knots,patterns}

\def\Z2{\mathbb{Z}_2^2}
\def\g{\mathfrak{g}}
\def\DP#1#2{\hat{#1}\cdot \hat{#2}}
\def\ph#1#2{(-1)^{\hat{#1}\cdot\hat{#2}}}
\def\f#1#2{f_{#1}^{\ \ #2}}
\def\tf#1#2{\tilde{f}_{#1}^{\ \ #2}}

\newtheorem{lemma}{Lemma}[section]
\newtheorem{prop}[lemma]{Proposition}
\newtheorem{thm}[lemma]{Theorem}
\newtheorem{DEF}[lemma]{Definition}
\newtheorem{cor}[lemma]{Corollary}

\begin{document}

\begin{titlepage}
	\title{Universal weight systems from a minimal $\mathbb{Z}_2^2$-graded Lie algebra}
	\date{\today}
	\author{N. Aizawa, \and Daichi Kimura}
	\maketitle
	\begin{center}
		\textit{Department of Physics, Graduate School of Science,
			\\
			Osaka Metropolitan University, 
			\\Nakamozu Campus, 
			Sakai, Osaka 599-8531 Japan}
	\end{center}
	
	\vfill
\begin{abstract}
Color Lie algebras, which were introduced by Ree, are a graded extension of Lie (super)algebras by an abelian group. 
We show that the color Lie algebras can be used to construct universal weight systems for knot invariants of of Vassiliev and Kontsevich. 
As a simple example, we take $\mathbb{Z}_2 \times \mathbb{Z}_2$ as the grading group  and consider the four-dimensional color Lie algebra called $A1_{\epsilon}$. 
The weight system constructed from $A1_{\epsilon}$ is studied in some detail and some relations between the weights, such as the recurrence relation for chord diagrams, are derived. 
These relations show that the weight system from $A1_{\epsilon}$ is a hybrid of those from $sl(2)$ and $gl(1|1)$.

\end{abstract}
	
	\thispagestyle{empty}
\end{titlepage}

\clearpage
\setcounter{page}{1}
\section{Introduction}

The purpose of the present work is to make a connection between the knot invariants of Vassiliev and Kontsevich and the color Lie algebras. The value of a Vassiliev invariant is determined once those of some singular knots are given, which in turn depends only on the chord diagram of the singular knots \cite{Va1}. 
The Kontsevich invariant is a quite powerful topological invariant because it is universal among the Vassiliev invariants, in the sense that any Vassiliev invariant is obtained from the Kontsevich invariant by some projection. The Kontsevich invariant belongs to the space of Jacobi diagrams, and the Jacobi and chord diagrams are related by the STU relation. Thus, if we want to evaluate the values of these invariants, we need a map from the chord and Jacobi diagrams to some numbers. Such map is called a \textit{weight system} and it is known that weight systems are constructed by using Lie (super)algebras. There are some good monographs on these subjects, we refer the reader to see them for more details \cite{Ohtsuki,ChDuMo,JacMo}.  

The color Lie algebra (a.k.a. $\epsilon$-Lie algebra) is a generalization of the Lie (super)algebras and was first introduced  by Ree in 1960 \cite{Ree}. It is a very general class of algebras and includes Lie and Lie superalgebras as subclasses. Later, in the 1970s, another subclass of  color Lie algebras was rediscovered by Rittenberg and Wyler \cite{rw1,rw2} (see also \cite{sch}). Since then, the algebraic structure has been studied continuously (see \cite{StoVDJ4,Ryan2} for recent developments) and that of  Rittenberg and Wyler has recently attracted renewed interest in connection with higher supergeometry (see \cite{Pz2nint,PonSch} for a concise review), new symmetries in physics \cite{aktt1,aktt2,Ryan1,Bruce,BruDup,Topp} and new integrable systems characterized by color Lie algebras \cite{bruSG,ait,niktt}.

We will see that color Lie algebras are used to construct weight system, which is one of our main results (Proposition \ref{PROP:ColorS}). 
In fact, this is easily seen by checking that the color Lie algebra satisfies the definition  of the $S$-Lie algebra introduced by Vaintrob \cite{Vaintrob}. The $S$-Lie algebra is defined in such a way that  it has sufficient algebraic structure to construct weight system. 
As an example, we take one of the minimal color Lie algebras $A1_{\epsilon}$ (introduced in \cite{FZminimal}) and consider the  universal weight system constructed from $A1_{\epsilon}$. 
The reason of this choice is that $A1_{\epsilon}$ is a simple and non-trivial example of color Lie algebras. It will be shown that the $A1_{\epsilon}$ weight system is a two-variable polynomial and we derive the recurrence relation of the weight systems (Theorem \ref{THM:recrel}).  
It then turns out that the $A1_{\epsilon}$ weight system is a hybrid of  the $sl(2)$ and $gl(1|1)$ weight systems investigated in \cite{ChVar} and \cite{FOKV}, respectively. 
Thus, weight systems obtained from color Lie algebras will be in general different from those of Lie (super)algebras and non-trivial.
 
A weight system coming from an algebra is said to be universal if it maps a Jacobi (chord) diagram to a center of the enveloping algebra. The universal weight systems allow us to have more invariants depending on the choice of a representation. However, computing  universal weight systems by the procedure of \cite{Vaintrob} becomes more difficult as the dimension of the algebra increases. 
If we want  more examples of the universal weight systems, color Lie algebras offer the possibility beyond Lie (super)algebras.  
Another approach to construct universal weight systems has been proposed and applied to $gl(n)$ and $ gl(n|m)$ \cite{Yang1,Yang2}. 

We plan this paper as follows: Most parts of the first two sections are a review. 
In \S \ref{SEC:SLie}, we give the definition of $S$-Lie algebras and explain how to construct weight systems from them. 
In \S \ref{SEC:colorLie}, the definition of color Lie algebras is given and we observe that the color Lie algebras are one of the examples of $S$-Lie algebras. This allows us to construct weight systems using color Lie algebras by the procedure described in \S \ref{SEC:SLie}. 
As an example of the color Lie algebras, we take a $\Z2$-graded Lie algebra which is called $ A1_{\epsilon} $ and show that it has a non-degenerate invariant bilinear form. 
We study the weight systems constructed from $A1_{\epsilon} $ for framed knots in details in \S \ref{SEC:A1WSframe}.  We present some relations which are useful for calculating the weight of a given diagram and they are proved in \S \ref{SEC:Proof}. 
In \S \ref{SEC:deframing}, we show that weight system for ordinary (unframed) knots are a specialization of that for framed knots. Finally, we summarize our results and give some remarks to conclude the paper in \S \ref{SEC:CR}. To have the invariant bilinear form for $A1_{\epsilon}$, a theory of invariant bilinear forms on $\Z2$-graded Lie algebras is developed in Appendix.

\section{$S$-Lie algebra weight system} \label{SEC:SLie}

This section is a brief review of the $S$-Lie algebra and the weight system constructed from it \cite{Vaintrob}.

\subsection{Definition of $S$-Lie algebras}

 The presentation of this subsection is slightly different from the original one in \cite{Vaintrob}.

\begin{DEF} \label{DEF:SLie}
	Let $L$ be a vector space over the field $\mathbb{K}$ of characteristic zero. 
	$L$ is referred to as a $S$-Lie algebra if it admits two bilinear mappings
	\begin{equation}
		f : L \times L \to L, \qquad S : L \times L \to L \times L
	\end{equation}
	satisfying the following identities
	\begin{enumerate}
		\item $S^2 = id_{L\otimes L}$
		\item $ S_{12} S_{23} S_{12} = S_{23} S_{12} S_{23}$
		\item $ S f_{12} = f_{23} S_{12} S_{23}$
		\item $ S f_{23} = f_{12} S_{23} S_{12}$
		\item $ f S = -f$
		\item $ f f_{12} = f f_{23} - f f_{23} S_{12}$
	\end{enumerate}
	where $ S_{12} $ and $ S_{23}$ stand for the action of $ S \otimes id$, $ id \otimes S$ on $ L \otimes L \otimes L, $ respectively.  $ f_{12}, f_{23}$ are understood similarly. 
\end{DEF} 

The relations in this definition are realized pictorially by Jacobi diagrams. 
A Jacobi diagram on the circle $S^1$ is a uni-trivalent graph (each  vertex is either univalent or trivalent) whose univalent vertices are distinct points on $S^1$. If $S^1$ has an orientation, the trivalent vertices are also oriented according to the orientation of $S^1$. 
Some examples are shown below:
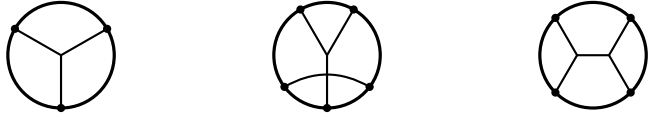
\begin{center}
	\begin{tikzpicture}[scale=1.4]
		\draw[thick] (0,-.5) coordinate(A) -- (0,0); 
			\fill (A) circle (1.1pt); 
		\draw[thick] (-.433,.25) coordinate(B) -- (0,0) -- (.433,.25) coordinate(C);
			\fill (B) circle (1.1pt); \fill (C) circle (1.1pt);
		\draw[very thick] (0,0) circle (.5);
		\begin{scope}[xshift=2.5cm]
			\draw[thick] (0,-.5) coordinate(A) -- (0,0);
			\draw[thick] (-.25,.433) coordinate(B) -- (0,0) -- (.25,.433) coordinate(C);
			\draw[thick] (.4,-.3) coordinate(D) to [out=150,in=30] (-.4,-.3) coordinate(E);
			\draw[very thick] (0,0) circle (.5);
			\foreach \label in {A,B,C,D,E} 
				\fill (\label) circle (1.1pt);
		\end{scope}
		\begin{scope}[xshift=5cm]
			\draw[thick] (-.15,0)  -- (.15,0) ;
			\draw[thick] (.3535,.3535) coordinate(A) -- (.15,0) -- (.3535,-.3535) coordinate(B);
			\draw[thick] (-.3535,.3535) coordinate(C) -- (-.15,0) -- (-.3535,-.3535) coordinate(D);
			\draw[very thick] (0,0) circle (.5);
			\foreach \label in {A,B,C,D}
				\fill (\label) circle (1.1pt);
		\end{scope}
	\end{tikzpicture}
\end{center}
The two bilinear mappings $ f $ and $ S $ are represented as follows:
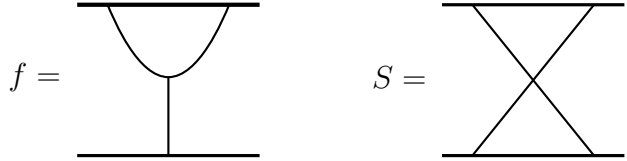
\begin{center}
	\begin{tikzpicture}[scale=0.8]
		\draw[ultra thick] (0,2.5) coordinate(ul)  -- (3,2.5) coordinate (ur);
		\draw[very thick] (0,0) coordinate(O) -- (3,0) coordinate(lr);
		\draw[thick] ($(ul)+(0.5,0)$) parabola[bend at end] (1.5,1.3) coordinate(v) parabola ($(ur)+(-0.5,0)$);
		\draw[thick] (v) -- (1.5,0);
		\node at ($(v)+(-2.2,0) $) {$f = $}; 
		\begin{scope}[xshift=6cm]
		  \draw[very thick] (0,2.5) coordinate(ul)  -- (3,2.5) coordinate (ur);
		  \draw[very thick] (0,0) coordinate(O) -- (3,0) coordinate(lr);	
		  \draw[thick] ($(ul)+(0.5,0)$)--($(lr)+(-0.5,0)$); 
		  \draw[thick] ($(ur)+(-0.5,0)$)--($(O)+(0.5,0)$);
		  \node at (-0.7,1.3) {$S = $};
		\end{scope}
		%
	\end{tikzpicture}
\end{center}
where two horizontal parallel lines are segments of $S^1$ and the  crossing of the mapping $S$ is not a vertex. 
The univalent vertex corresponds to an element of $L$ and the trivalent vertex represents a product of two elements of $L$. 
The composition of two maps is represented by connecting two graphs:
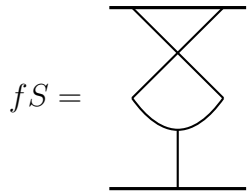
\begin{center}
	\begin{tikzpicture}[scale=0.6]
		\draw[very thick] (0,4) coordinate(ul)  -- (3,4) coordinate (ur);
		\draw[opacity=0] (0,2) coordinate(O) -- (3,2) coordinate(lr);	
		\draw[thick] ($(ul)+(0.5,0)$)--($(lr)+(-0.5,0)$); 
		\draw[thick] ($(ur)+(-0.5,0)$)--($(O)+(0.5,0)$);
		\draw[thick] ($(ul)+(0.5,-2)$) parabola[bend at end] (1.5,1.3) coordinate(v) parabola ($(ur)+(-0.5,-2)$);
		\draw[thick] (v) -- (1.5,0);
		\draw[very thick] (0,0)--(3,0);
		\node at ($(O)+(-1.4,0) $) {$f S= $}; 
	\end{tikzpicture}
\end{center}
The defining relations of the $S$-Lie algebra are shown as follows:
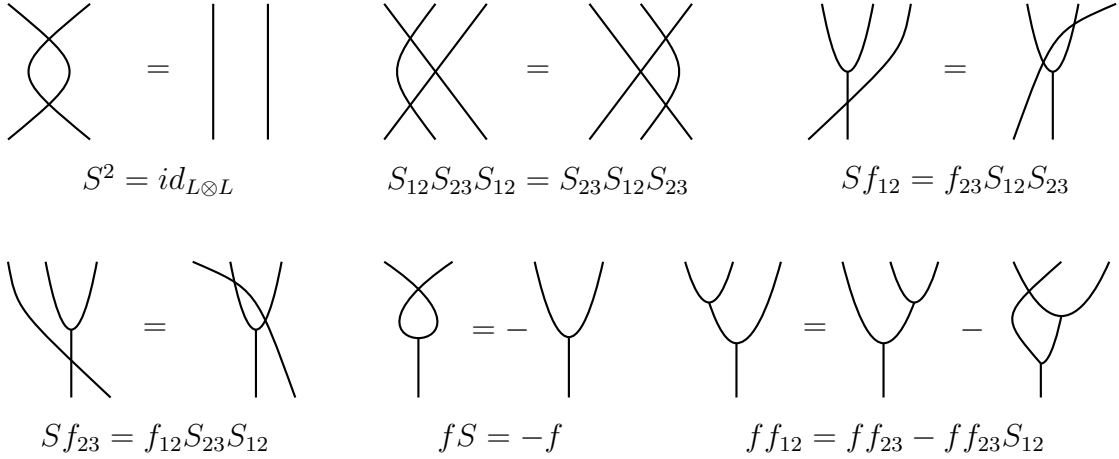
\begin{center}
	\begin{tikzpicture}[scale=0.9]
		%
		%
		\coordinate (ul) at (0,2); \coordinate (ur) at (1.2,2);
		\coordinate (ll) at (0,0); \coordinate (lr) at (1.2,0); 
		\draw[thick] (ul) .. controls($(ul)!0.5!(ll)+(1.2,0)$) .. (ll); 
		\draw[thick] (ur) .. controls($(ur)!0.5!(lr)+(-1.2,0)$) .. (lr); 
		\node at ($(ur)!0.5!(lr)+(1,0)$) {$=$};
		\node at ($(ur)!0.5!(lr)+(1,-1.6)$) {$S^2=id_{L\otimes L}$};
		\begin{scope}[xshift=3cm]
			\coordinate (ul) at (0,2); \coordinate (ur) at (0.8,2);
			\coordinate (ll) at (0,0); \coordinate (lr) at (0.8,0);
			\draw[thick] (ul)--(ll); \draw[thick] (ur)--(lr); 
		\end{scope}
	    %
		%
		\begin{scope}[xshift=5.5cm]
			\coordinate (u1) at (0,2); \coordinate (u2) at (0.75,2); \coordinate (u3) at (1.5,2);
			\coordinate (l1) at (0,0); \coordinate (l2) at (0.75,0); \coordinate (l3) at (1.5,0);
			\draw[thick] (u1)--(l3); \draw[thick] (u3)--(l1);
			\draw[thick] (u2) ..controls($(u2)!0.5!(l2)+(-0.75,0)$)..(l2);
		\end{scope}
		\node at ($(u2)!0.5!(l2)+(1.5,0)$) {$=$};
		\node at ($(u2)!0.5!(l2)+(1.5,-1.6)$) {$ S_{12} S_{23} S_{12} = S_{23} S_{12} S_{23}$};
		\begin{scope}[xshift=8.5cm]
			\coordinate (u1) at (0,2); \coordinate (u2) at (0.75,2); \coordinate (u3) at (1.5,2);
			\coordinate (l1) at (0,0); \coordinate (l2) at (0.75,0); \coordinate (l3) at (1.5,0);
			\draw[thick] (u1)--(l3); \draw[thick] (u3)--(l1);
			\draw[thick] (u2) ..controls($(u2)!0.5!(l2)+(+0.75,0)$)..(l2);
		\end{scope}
		%
		%
		\begin{scope}[xshift=11.9cm]
			\coordinate (u1) at (0,2); \coordinate (u2) at (0.75,2); \coordinate (u3) at (1.3,2);
			\coordinate (l1) at (0,0); \coordinate (l2) at (0.75,0); \coordinate (l3) at (1.3,0);
			\coordinate (v) at (0.375,1);
			\draw[thick] (u1) parabola[bend at end] (v) parabola (u2);
			\draw[thick] (v)--(0.375,0);
			\draw[thick] (u3) ..controls($(v)+(0.8,0.3)$).. ($(l1)+(-0.2,0)$);
		\end{scope}
		\node at ($(u2)!0.5!(l2)+(1.2,0)$) {$=$};
		\node at ($(u2)!0.5!(l2)+(1.2,-1.6)$) {$ S f_{12} = f_{23} S_{12} S_{23}$};
		\begin{scope}[xshift=14.9cm]
			\coordinate (u1) at (0,2); \coordinate (u2) at (0.75,2); \coordinate (u3) at (1.3,2);
			\coordinate (l1) at (0,0); \coordinate (l2) at (0.75,0); \coordinate (l3) at (1.3,0);
			\coordinate (v) at (0.375,1);
			\draw[thick] (u1) parabola[bend at end] (v) parabola (u2);
			\draw[thick] (v)--(0.375,0);
			\draw[thick] (u3) ..controls($(v)+(0,0.6)$)..  ($(l1)+(-0.2,0)$);
		\end{scope}
		%
		%
		\begin{scope}[yshift=-3.8cm]
			\coordinate (u1) at (0,2); \coordinate (u2) at (0.55,2); \coordinate (u3) at (1.3,2);
			\coordinate (l1) at (0,0); \coordinate (l2) at (0.55,0); \coordinate (l3) at (1.3,0);
			\coordinate (v) at (0.925,1);
			\draw[thick] (u2) parabola[bend at end] (v) parabola (u3);
			\draw[thick] (v) -- (0.925,0);
			\draw[thick] (u1) ..controls($(v)+(-0.8,0.3)$).. ($(l3)+(0.2,0)$);
		\end{scope}
		\node at ($(u2)!0.5!(l2)+(1.6,0)$) {$=$};
		\node at ($(u2)!0.5!(l2)+(1.6,-1.6)$) {$S f_{23} = f_{12} S_{23} S_{12}$};
		\begin{scope}[xshift=2.7cm,yshift=-3.8cm]
			\coordinate (u1) at (0,2); \coordinate (u2) at (0.55,2); \coordinate (u3) at (1.3,2);
			\coordinate (l1) at (0,0); \coordinate (l2) at (0.55,0); \coordinate (l3) at (1.3,0);
			\coordinate (v) at (0.925,1);
			\draw[thick] (u2) parabola[bend at end] (v) parabola (u3);
			\draw[thick] (v) -- (0.925,0);
			\draw[thick] (u1) ..controls($(v)+(0,0.6)$).. ($(l3)+(0.2,0)$);
		\end{scope}
		%
		%
		\begin{scope}[xshift=5.5cm, yshift=-3.8cm]
			\coordinate (u1) at (0,2); \coordinate (u2) at (1,2);
			\coordinate (l1) at (0,0); \coordinate (l2) at (1,0);
			\draw[thick] (u1) ..controls ($(u1)+(2.2,-1.5)$) and ($(u1)+(-1.2,-1.5)$).. (u2);
			\draw[thick] (0.5,0.89)--(0.5,0);
		\end{scope}
		\node at ($(u2)!0.5!(l2)+(0.7,0)$) {$=-$};
		\node at ($(u2)!0.5!(l2)+(0.7,-1.6)$) {$fS =-f$};
		\begin{scope}[xshift=7.7cm, yshift=-3.8cm]
			\coordinate (u1) at (0,2); \coordinate (u2) at (1,2);
			\coordinate (l1) at (0,0); \coordinate (l2) at (1,0);
			\coordinate (v) at (0.5,0.89);
			\draw[thick] (u1) parabola[bend at end] (v) parabola (u2);
			\draw[thick] (0.5,0.89)--(0.5,0);
		\end{scope}
		%
		%
		\begin{scope}[xshift=9.9cm,yshift=-3.8cm]
			\coordinate (u1) at (0,2); \coordinate (u2) at (0.7,2); \coordinate (u3) at (1.4,2);
			\coordinate (l1) at (0,0); \coordinate (l2) at (0.7,0); \coordinate (l3) at (1.4,0);
			\coordinate (v1) at (0.35,1.4);  \coordinate (v2) at (0.75,0.8);
			\draw[thick] (u1) parabola[bend at end] (v1) parabola (u2);
			\draw[thick] (v1) parabola[bend at end] (v2) parabola (u3);
			\draw[thick] (v2)--($(v2)+(0,-0.8)$);
		\end{scope}	
		\node at ($(u2)!0.5!(l2)+(1.2,0)$) {$=$};
		\node[xshift=30] at ($(u2)!0.5!(l2)+(1.2,-1.6)$) {$ f f_{12} = f f_{23} - f f_{23} S_{12}$};
		\begin{scope}[xshift=12.2cm,yshift=-3.8cm]
			\coordinate (u1) at (0,2); \coordinate (u2) at (0.7,2); \coordinate (u3) at (1.4,2);
			\coordinate (l1) at (0,0); \coordinate (l2) at (0.7,0); \coordinate (l3) at (1.4,0);
			\coordinate (v1) at (0.6,0.8);  \coordinate (v2) at (1.05,1.4);
			\draw[thick] (u2) parabola[bend at end] (v2) parabola (u3);
			\draw[thick] (u1) parabola[bend at end] (v1) parabola (v2);
			\draw[thick](v1)--($(v1)+(0,-0.8)$);
		\end{scope}
		\node at ($(u2)!0.5!(l2)+(1.2,0)$) {$-$};
		\begin{scope}[xshift=14.7cm,yshift=-3.8cm]
			\coordinate (u1) at (0,2); \coordinate (u2) at (0.7,2); \coordinate (u3) at (1.4,2);
			\coordinate (l1) at (0,0); \coordinate (l2) at (0.7,0); \coordinate (l3) at (1.4,0);
			\coordinate (v1) at (0.4,0.5);  \coordinate (v2) at (0.7,1.2);
			\draw[thick] (u1) parabola[bend at end] (v2) parabola (u3);
			\draw[thick] (v1) parabola  (v2);
			\draw[thick] (u2) .. controls($(v2)+(-0.9,0)$).. (v1);
			\draw[thick] (v1)--($(v1)+(0,-0.5)$);
		\end{scope}	
	\end{tikzpicture}
\end{center}
It is seen that the last two relations are AS and IHX relations, respectively. They are implemented in the definition of the $S$-Lie algebra. 

Denoting a basis of $L$ by $ X_i, i = 1, 2, \dots, \dim L $ and writing 
\begin{equation}
	f(X_i,X_j) = \sum_k \f{ij}{k} X_k, \qquad 
	S(X_i,X_j) = \sum_{k,\ell} S_{ij}^{k\ell} X_k X_{\ell}, \label{fScomp}
\end{equation} 
the constants $ \f{ij}{k},\; S_{ij}^{k\ell} \in \mathbb{K} $ are assigned to a trivalent vertex and a crossing:
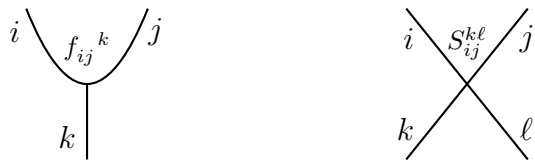
\begin{center}
	\begin{tikzpicture}
		\coordinate(u1) at (0,2); \coordinate (u2) at (1.6,2); 
		\coordinate (l) at (0.8,0); 
		\coordinate (v) at (0.8,1);
		\draw[thick] (u1) parabola[bend at end] (v) parabola (u2);
		\draw[thick] (v)--(l);
		\node[left, xshift=-5] at($(u1)!0.3!(v)$) {$i$};
		\node[right,xshift=3] at($(u2)!0.3!(v)$) {$j$};
		\node[left] at($(v)!0.7!(l)$) {$k$};
		\node[above,yshift=3] at (v)  {\footnotesize $\f{ij}{k}$}; 
		\begin{scope}[xshift=5cm]
			\coordinate(u1) at (0,2); \coordinate (u2) at (1.6,2); 
			\coordinate(l1) at (0,0); \coordinate (l2) at (1.6,0); 
			\draw[thick] (u1)--(l2); \draw[thick] (u2)--(l1);
			\node[left,xshift=-2] at ($(u1)!0.2!(l2)$) {$i$}; 
			\node[right,xshift=2] at ($(u2)!0.2!(l1)$) {$j$};
			\node[right,xshift=2] at ($(u1)!0.8!(l2)$) {$\ell$}; 
			\node[left,xshift=-2] at ($(u2)!0.8!(l1)$) {$k$};
			\node[above,yshift=5] at ($(u1)!0.5!(l2)$) {\footnotesize $S_{ij}^{k\ell}$};
		\end{scope}
	\end{tikzpicture}
\end{center}
and each edge is labelled by the basis. 

A bilinear form $ B : L \times L \to \mathbb{K}$ is represented by
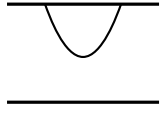
\begin{center}
	\begin{tikzpicture}
		\coordinate (A) at (0,2);  \coordinate (B) at (1,2);
		\coordinate (v) at (0.5,1.3);
		\draw[very thick] ($(A)+(-0.5,0)$) -- ($(B)+(0.5,0)$);
		\draw[very thick] ($(A)+(-0.5,-1.3)$) -- ($(B)+(0.5,-1.3)$);
		\draw[thick] (A) parabola[bend at end] (v) parabola (B);
	\end{tikzpicture}
\end{center}
and $ B_{ij} := B(X_i,X_j).$ 
If the bilinear form $B$ is non-degenerate, its inverse is denoted by $ C^{ij}$:
\begin{equation}
	\sum_k B_{ik} C^{kj} = \sum_k C^{jk} B_{ki} = \delta^i_j
\end{equation}
and graphically
\begin{center}
	\begin{tikzpicture}[scale=0.9]
		\begin{scope}[xshift=-20]
		\node at (-4,1) {$ C^{ij} = $};
		\draw[thick] (-2.8,0.7) node[left] {$i$} parabola[bend at end] (-2.3,1.5) parabola (-1.8,0.7) node[right] {$j$,};
		\end{scope}
		\node (u) at (0,2) {$i\ \; $}; \node (l) at (1.7,0) {$\ \; j$};
		\coordinate (v1) at (0.5,0.6); \coordinate (v2) at (0.8,1); \node (v3) at (1.2,1.6) {$k\qquad$};
		\draw[thick] (u) parabola[bend at end] (v1) parabola (v2) parabola[bend at end] (v3) parabola (l);
		\node at ($(l)+(0.5,1)$) {$=$};
		\begin{scope}[xshift=2.8cm]
			\node (u) at (0,0) {$j \ \;$}; \node (l) at (1.7,2) {$\ \;i$};
			\node (v1) at (0.6,1.6) {$\qquad k$}; \coordinate (v2) at (0.9,1.2); \coordinate (v3) at (1.2,0.8);
			\draw[thick] (u) parabola[bend at end] (v1) parabola (v2) parabola[bend at end] (v3) parabola (l);
        \end{scope}
        \node at ($(l)+(0.5,-1)$) {$=$};
        \begin{scope}[xshift=5.6cm]
        	\draw[thick] (0,0) node[right] {$j$} -- (0,2) node[right] {$i$};
        \end{scope}
	\end{tikzpicture}
\end{center}

\begin{DEF}
	The universal enveloping algebra $\mathcal{U}(L)$ of the $S$-Lie algebra $L$ is defined as the quotient algebra of the tensor algebra 
	\begin{equation}
		T^*(L) := \bigoplus_{n \in \mathbb{Z}_{\geq 0}} L^{\otimes n}
	\end{equation}
	by the ideal generated by the expression
	\begin{equation}
		X \otimes Y - S(X,Y) - f(X,Y), \ X, Y \in L  \label{ideal}
	\end{equation}
\end{DEF}

\noindent
The graphical expression of \eqref{ideal} is given by
\begin{equation}
	\begin{tikzpicture}[baseline=0pt]
		\coordinate (u1) at (0,1); \coordinate (u2) at (1,1);
		\coordinate (l1) at (0,0); \coordinate (l2) at (1,0);
		\draw[very thick] ($(l1)+(-0.5,0)$)--($(l2)+(0.5,0)$);
		\draw[thick] (u1)--(l1); \draw[thick] (u2)--(l2); 
		\node[below] at (l1) {\footnotesize $X_i$}; \node[below] at (l2) {\footnotesize $X_j$};
		\node at ($(l2)!0.5!(u2)+(1,0)$) {$-$}; 
		\begin{scope}[xshift=3cm]
			\coordinate (u1) at (0,1); \coordinate (u2) at (1,1);
			\coordinate (l1) at (0,0); \coordinate (l2) at (1,0);
			\draw[very thick] ($(l1)+(-0.5,0)$)--($(l2)+(0.5,0)$);
			\draw[thick] (u1)--(l2); \draw[thick] (u2)--(l1); 
			\node[below] at (l1) {\footnotesize $X_i$}; \node[below] at (l2) {\footnotesize $X_j$};
			\node at ($(l2)!0.5!(u2)+(1,0)$) {$-$}; 
		\end{scope}

		\begin{scope}[xshift=6cm]
			\coordinate (u1) at (0,1); \coordinate (u2) at (1,1);
			\coordinate (l1) at (0,0); \coordinate (l2) at (1,0);
			\coordinate (v) at ($(u1)!0.5!(u2)+(0,-0.5)$);
			\draw[very thick] ($(l1)+(-0.5,0)$)--($(l2)+(0.5,0)$);
			\draw[thick] (u1) parabola[bend at end](v) parabola (u2);
			\draw[thick] (v)--($(v)+(0,-0.5)$); 
			\node[below] at ($(v)+(0,-0.5)$) {\footnotesize $f(X_i,X_j)$}; 
			\node at ($(l2)!0.5!(u2)+(1,0)$) {$= 0$}; 
		\end{scope}
	\end{tikzpicture}
	\label{STUpic}
\end{equation}
and it is the STU relation.

Using the definitions above, the following is easy to verify.
\begin{prop} \label{PROP:S-Casi}
	If the bilinear form $ B$ on $L$ is non-degenerate and satisfies 
	\begin{enumerate}
		\item $ B f_{12} = B f_{23}$
		\item $ B_{12} S_{23} = B_{23} S_{12}$
	\end{enumerate}
	then the 2nd order Casimir of $L$ is given by
	\begin{equation}
		c = \sum_{i,j} C^{ij} X_i X_j, 
	\end{equation}
	namely, $ c X_i = X_i c $ for any $ X_i \in L.$ 
\end{prop}

The conditions in this proposition are presented graphically as 
\begin{center}
	\begin{tikzpicture}
		\coordinate (u1) at (0,1.2); \coordinate (u2) at (0.8,1.2); \coordinate (u3) at (1.6,1.2);
		\coordinate (l1) at (0,0);  \coordinate (l2) at (0.8,0); \coordinate (l3) at (1.6,0);
		\draw[thick]  (u1) parabola[bend at end] (l2) parabola (u3);
		\draw[thick] (u2)--(l3);
		\node at ($(u3)!0.5!(l3)+(0.5,0)$) {$=$};
		\node at ($(u3)!0.5!(l3)+(0.5,-1.2)$) {$B_{12} S_{23} = B_{23} S_{12}$};
		\begin{scope}[xshift=2.5cm]
			\coordinate (u1) at (0,1.2); \coordinate (u2) at (0.8,1.2); \coordinate (u3) at (1.6,1.2);
			\coordinate (l1) at (0,0);  \coordinate (l2) at (0.8,0); \coordinate (l3) at (1.6,0);
			\draw[thick]  (u1) parabola[bend at end] (l2) parabola (u3);
			\draw[thick] (u2)--(l1);
		\end{scope}
		\begin{scope}[xshift=6cm]
			\coordinate (u1) at (0,1.2); \coordinate (u2) at (1,1.2); \coordinate (u3) at (1.6,1.2);
			\coordinate (v1) at (0.5,0.4); \coordinate (v2) at (0.9,0);
			\draw[thick] (u1) parabola[bend at end] (v1) parabola (u2);
			\draw[thick] (v1) parabola[bend at end] (v2) parabola (u3);
			\node at ($(u2)+(1,-0.6)$) {$=$};
			\node at ($(u2)+(1,-1.8)$) {$Bf_{12}= B f_{23}$};
		\end{scope}
		\begin{scope}[xshift=8.5cm]
			\coordinate (u1) at (0,1.2); \coordinate (u2) at (0.6,1.2); \coordinate (u3) at (1.6,1.2);
			\coordinate (v1) at (0.7,0); \coordinate (v2) at (1.1,0.4);
			\draw[thick] (u2) parabola[bend at end] (v2) parabola (u3);
			\draw[thick] (u1) parabola[bend at end] (v1) parabola (v2);
		\end{scope}
	\end{tikzpicture}
\end{center}
In the construction of weight systems, we need the conditions rewritten in terms of $C^{ij}$:
\begin{align}
	\sum_k S_{jk}^{mn} C^{k\ell} = \sum_k  C^{mk} S_{kj}^{n\ell}, 
	\qquad
	\sum_k C^{mk} \f{kj}{n} = \sum_k \f{jk}{m} C^{kn}
\end{align}
and their graphical representation
\begin{center}
	\begin{tikzpicture}
		\coordinate (u1) at (0,1.2); \coordinate (u2) at (0.8,1.2); \coordinate (u3) at (1.6,1.2);
		\coordinate (l1) at (0,0);  \coordinate (l2) at (0.8,0); \coordinate (l3) at (1.6,0);
		\draw[thick]  (l1) parabola[bend at end] (u2) parabola (l3);
		\draw[thick] (u1)--(l2);
		\node at ($(u3)!0.5!(l3)+(0.5,0)$) {$=$};
		\begin{scope}[xshift=2.5cm]
			\coordinate (u1) at (0,1.2); \coordinate (u2) at (0.8,1.2); \coordinate (u3) at (1.6,1.2);
			\coordinate (l1) at (0,0);  \coordinate (l2) at (0.8,0); \coordinate (l3) at (1.6,0);
			\draw[thick]  (l1) parabola[bend at end] (u2) parabola (l3);
			\draw[thick] (l2)--(u3);
		\end{scope}
		\begin{scope}[xshift=6cm]
			\coordinate (l1) at (0,0); \coordinate (u1) at (0.45,1.2); \coordinate (u2) at (1.6,1.2);
			\coordinate (v1) at (0.7,0.85); \coordinate (v2) at (1.1,0.4);
			\draw[thick] (l1) parabola[bend at end] (u1) parabola (v1)  parabola[bend at end] (v2) parabola (u2);
			\draw[thick] (v2)--($(v2)+(0,-0.4)$);
			\node at ($(u2)+(0.5,-0.6)$) {$=$};
		\end{scope}
		\begin{scope}[xshift=9cm]
			\coordinate (l1) at (1.6,0); \coordinate (u1) at (0,1.2); \coordinate (u2) at (1.1,1.2);
			\coordinate (v1) at (0.4,0.4); \coordinate (v2) at (0.8,0.85);
			\draw[thick] (u1) parabola[bend at end] (v1) parabola (v2) parabola[bend at end] (u2) parabola (l1);
			\draw[thick] (v1)--($(v1)+(0,-0.4)$);
		\end{scope}
	\end{tikzpicture}
\end{center}

A Lie algebra gives a simple example of an $S$-Lie algebra where $ S(X_i,X_j) = X_j \otimes X_i$ and the mapping $f$ is the Lie bracket. It is also easy to see that a Lie superalgebra is an $S$-Lie algebra. We take $S$ as $X_i, X_j, $  $ S(X_i,X_j) = (-1)^{|X_i|\,|X_j|} X_j \otimes X_i $ for homogeneous elements with $\mathbb{Z}_2$-gradings $|X_i|$ and $ |X_j|. $ The mapping $f$ is the super Lie bracket.

\subsection{Universal weight systems from $S$-Lie algebras} \label{SEC:UWSL}

First, we collect some  of the basics of weight systems. 
Recall that a chord diagram is a Jacobi diagram without trivalent vertices:
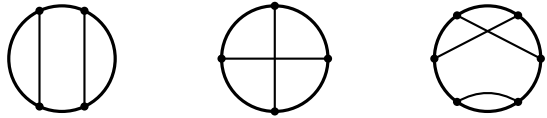
\begin{center}
	\begin{tikzpicture}[scale=1.4]
		\draw[thick] (65:.5) coordinate(A) -- (-65:.5) coordinate(B);
		\draw[thick] (115:.5) coordinate(C) -- (245:.5) coordinate(D);
		\draw [very thick] (0,0) circle (.5);
		\foreach \label in {A,B,C,D}
		\fill (\label) circle (1.1pt);
		\begin{scope}[xshift=2cm]
			\draw[thick] (0,.5) coordinate(A) -- (0,-.5) coordinate(B);
			\draw[thick] (.5,0) coordinate(C) -- (-.5,0) coordinate(D);
			\draw [very thick] (0,0) circle (.5);
			\foreach \label in {A,B,C,D}
			\fill (\label) circle (1.1pt);
		\end{scope}
		\begin{scope}[xshift=4cm]
			\draw[thick] (.5,0) coordinate(A) -- (125:.5) coordinate(B);
			\draw[thick] (-.5,0) coordinate(C) -- (55:.5) coordinate(D);
			\draw[thick] (-55:.5) coordinate(E) to [out=150,in=30] (235:.5) coordinate(F);
			\draw [very thick] (0,0) circle (.5);
			\foreach \label in {A,B,C,D,E,F}
			\fill (\label) circle (1.1pt);
		\end{scope}
	\end{tikzpicture}
\end{center}
A Jacobi diagram with $2n$ vertices (univalent and trivalent) resolves to  chord diagrams with $n$ chords by the STU relation \eqref{STUpic}. We call the number $n$ the \textit{order} of a diagram. 
The STU relation for Jacobi diagrams ensures the four-term (4T) relation for chord diagrams:
\begin{equation}
	\begin{tikzpicture}[baseline=0,scale=0.7]
		\draw[dotted,thick] (0,0) circle [radius=1.2];
		\foreach \t/\la in {90/A1,210/B1,330/C1}{
			\draw[very thick] (\t-30:1.2) arc (\t-30:\t+30:1.2); 
			\coordinate(\la) at ({1.2*cos(\t+15)},{1.2*sin(\t+15)});
		}
		\foreach \t/\lb in {90/A2,210/B2,330/C2}{
			\coordinate(\lb) at ({1.2*cos(\t-15)},{1.2*sin(\t-15)});
		}
		\draw[thick] (A2)  to[bend right] (C2);
		\draw[thick] (B1) to[bend left] (C1);
		\foreach \p in {A2,C2,B1,C1}
			\fill (\p) circle (1.8pt); 
		\node at (1.8,0) {$-$};
		\begin{scope}[xshift=3.5cm]
			\draw[dotted,thick] (0,0) circle [radius=1.2];
			\foreach \t/\la in {90/A1,210/B1,330/C1}{
				\draw[very thick] (\t-30:1.2) arc (\t-30:\t+30:1.2); 
				\coordinate(\la) at ({1.2*cos(\t+15)},{1.2*sin(\t+15)});
			}
			\foreach \t/\lb in {90/A2,210/B2,330/C2}{
				\coordinate(\lb) at ({1.2*cos(\t-15)},{1.2*sin(\t-15)});
			}
			\draw[thick] (A2) to[bend right] (C1);
			\draw[thick] (B1) to[bend left] (C2);
			\foreach \p in {A2,C1,B1,C2}
				\fill (\p) circle (1.8pt);
			\node at (1.8,0) {$+$};
		\end{scope}
		\begin{scope}[xshift=7cm]
			\draw[dotted,thick] (0,0) circle [radius=1.2];
			\foreach \t/\la in {90/A1,210/B1,330/C1}{
				\draw[very thick] (\t-30:1.2) arc (\t-30:\t+30:1.2); 
				\coordinate(\la) at ({1.2*cos(\t+15)},{1.2*sin(\t+15)});
			}
			\foreach \t/\lb in {90/A2,210/B2,330/C2}{
				\coordinate(\lb) at ({1.2*cos(\t-15)},{1.2*sin(\t-15)});
			}
			\draw[thick] (A2) to[bend right] (C1);
			\draw[thick] (A1) to[bend left] (B2);
			\foreach \p in {A2,C1,A1,B2}
				\fill (\p) circle (1.8pt);
			\node at (1.8,0) {$-$};
		\end{scope}
		\begin{scope}[xshift=10.5cm]
			\draw[dotted,thick] (0,0) circle [radius=1.2];
			\foreach \t/\la in {90/A1,210/B1,330/C1}{
				\draw[very thick] (\t-30:1.2) arc (\t-30:\t+30:1.2); 
				\coordinate(\la) at ({1.2*cos(\t+15)},{1.2*sin(\t+15)});
			}
			\foreach \t/\lb in {90/A2,210/B2,330/C2}{
				\coordinate(\lb) at ({1.2*cos(\t-15)},{1.2*sin(\t-15)});
			}
			\draw[thick] (A1) to[bend right] (C1);
			\draw[thick] (A2) to[bend left] (B2);
			\foreach \p in {A1,C1,A2,B2}
				\fill (\p) circle (1.8pt);
			\node at (1.8,0) {$= 0$};
		\end{scope}
	\end{tikzpicture}
	\label{4Tpic}
\end{equation}
 
Let $ \mathcal{A}_n $ (resp. $\bar{\mathcal{A}_n}$) be the space of formal linear combinations of order $n$ chord diagrams modulo 4T relations \eqref{4Tpic} (resp. 4T and one-term (1T) relations \eqref{1Tpic}). 
\begin{equation}
	\begin{tikzpicture}[baseline=0,scale=0.7]
		\draw[dotted,thick] (0,0) circle [radius=1.2];
		\draw[very thick] (130:1.2) arc (130:230:1.2);
		\coordinate (A) at ({1.2*cos(150)},{1.2*sin(150)}); 
		\coordinate (B) at ({1.2*cos(210)},{1.2*sin(210)});
		\draw[thick] (A)  to[bend left] (B);
		\fill (A) circle (1.8pt);  \fill (B) circle (1.8pt); 
		\node at (1.8,0) {$=0$}; 
	\end{tikzpicture}
	\label{1Tpic}
\end{equation}	
The corresponding graded spaces
\begin{equation}
	\mathcal{A} = \mathcal{A}_0 \oplus \mathcal{A}_1 \oplus \mathcal{A}_2 \oplus \dots, 
	\qquad
	\bar{\mathcal{A}} = \bar{\mathcal{A}}_0 \oplus \bar{\mathcal{A}}_1 \oplus \bar{\mathcal{A}}_2 \oplus \dots
\end{equation}
have a Hopf algebra structure where the product of two chord diagrams is their connected sum \cite{BN1}. 
A function $w_{\mathbb{K}}$ (resp. $\bar{w}_{\mathbb{K}}$) on $\mathcal{A}_n $ (resp. $\bar{\mathcal{A}}_n$) which takes its value in $\mathbb{K}$ is called a \textit{$\mathbb{K}$-valued weight system} of order $n$. 
This means that the weight system is the function that respects the 4T  (and 1T) relations.  	
We denote the set of all weight systems of order $n$ by $ \mathcal{W}_n $ (and $ \overline{\mathcal{W}}_n$ for $\bar{w}_{\mathbb{K}}$). 	
The fundamental theorem of weight systems is due to Kontsevich and Bar-Natan. 
\begin{thm}[\cite{BN1},\cite{Kon}]
	Let $\mathcal{V}_n$ and $\overline{\mathcal{V}}_n$ be the space of the Vassiliev invariants of order $n$ of framed and ordinary (unframed) knots, respectively. 
	\begin{enumerate}
		\item $ \mathcal{V}_n/\mathcal{V}_{n-1} \simeq \mathcal{W}_n \simeq \mathcal{A}_n^*$
		\item $ \overline{\mathcal{V}}_n/\overline{\mathcal{V}}_{n-1} \simeq \overline{\mathcal{W}}_n \simeq \bar{\mathcal{A}}_n^*$
	\end{enumerate}
	where $ \mathcal{A}_n^* $ (resp. $\bar{\mathcal{A}}_n^*$) is the space dual to $ \mathcal{A}_n$ (resp. $ \bar{\mathcal{A}}_n$). 
\end{thm}
There is an algebraic way to construct the weight systems, first developed by Kontsevich for simple Lie algebras \cite{Kon} and then generalized by Vaintrob for $S$-Lie algebras \cite{Vaintrob}. We define the \textit{universal weight system} of order $n$ by $ w : \mathcal{A}_n \to Z(\mathcal{U}(L))$ where $ Z(\mathcal{U}(L)) $ is the center of the universal enveloping algebra of the $S$-Lie algebra. The map $w$, of course, respects the 4T relations. Similarly, one may define $ \bar{w} : \bar{\mathcal{A}}_n \to Z(\mathcal{U}(L)). $ 
If we have a finite dimensional representation of the $S$-Lie algebra, then taking a trace of the image of $w$ (or $\bar{w}$) we get a $\mathbb{K}$-valued weight system. 
Before illustrating the construction of the universal weight systems for $S$-Lie algebra, we present an important theorem:
\begin{thm}[\cite{BN1}]
	Let $\mathcal{B}_n$ be the vector space spanned by Jacobi diagrams of order $n$ modulo the STU relation \eqref{STUpic}. Then, $ \mathcal{B}_n \simeq \mathcal{A}_n$
\end{thm}	
It follows that the space $ \mathcal{B} := \bigoplus_{n} \mathcal{B}_n$ also has a Hopf algebra structure with connected sum as product and a universal weight system can be defined similarly: $\mathcal{B} \to Z(\mathcal{U}(L)) $ which is a function taking its value in the center and respecting the STU relation.  
	
Now, we illustrate a construction of the universal weight system from an $S$-Lie algebra. 
We will discuss only on framed knots in this subsection and the treatment of unframed knots will be presented in \S \ref{SEC:deframing}. Consider the following diagram $D$. 
Cut $D$ at some point on $S^1$ and open it, then we get the diagram on a line segment (middle). 
Redraw the chords as a combination of the graphical presentations of the operations $f, S, B $ and $C$ of a $S$-Lie algebra $L$ (right). 
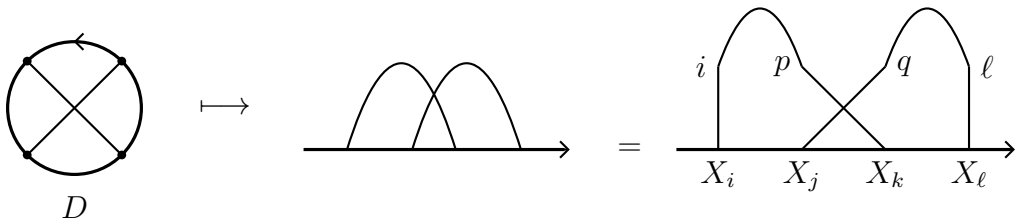
\begin{center}
	\begin{tikzpicture}[scale=1.1]
		\draw[thick] (45:.8) -- (225:.8); 
		 \fill (45:.8) circle (1.5pt);
		 \fill (225:.8) circle (1.5pt);
		\draw[thick] (135:.8) -- (-45:.8);
		  \fill (135:.8) circle (1.5pt);
		  \fill (-45:.8) circle (1.5pt);
		\draw [very thick] (0,0) circle (.8);
		\draw[-Straight Barb,thick] (.01,.8) -- (-.01,.8);
		\node at ($(90:.8)+(0,-2)$) {$D$};
		\node at (1.8,0) {$\longmapsto$};
		\begin{scope}[xshift=4.3cm,yshift=-0.5cm,scale=1.3]
			\draw[very thick] (-1.2,0) -- (1.2,0);
			\draw[thick] (-.8,0) parabola bend (-.3,.8) (.2,0);
			\draw[thick] (-.2,0) parabola bend (.3,.8) (.8,0);
			\draw[thick,-Straight Barb] (1.1,0) -- (1.25,0);
			\node at(1.8,0) {$=$};
		\end{scope}
		\begin{scope}[xshift=9.2cm]
			\draw[very thick] (-2,-.5) -- (2,-.5);
			\draw[-Straight Barb,thick] (1.9,-.5) -- (2.05,-.5);
			\draw[thick] (-.5,-.5) node[below] {$X_j$} -- (.5,.5) node[right] {$q$};
			\draw[thick] (-.5,.5) node[left] {$p$}-- (.5,-.5) node[below] {$X_k$};
			\draw[thick] (1.5,-.5) node[below] {$X_{\ell}$} -- (1.5,.5) node[right] {$\ell$};
			\draw[thick] (-1.5,-.5) node[below] {$X_i$} -- (-1.5,.5) node[left] {$i$};
			\draw[thick] (-1.5,.5) parabola bend (-1,1.2) (-.5,.5);
			\draw[thick] (1.5,.5) parabola bend (1,1.2) (.5,.5);
		\end{scope}
	\end{tikzpicture}
\end{center}
Assign the basis of $L$ to the univalent vertices. Taking a product of the basis along the direction of the line segment and summing over all the  indices, we obtain the associated element of $Z(\mathcal{U}(L))$:
\begin{equation}
	w(D) = \sum_{i,j,k,\ell,p,q} C^{ip}C^{q\ell}S_{pq}^{jk}X_iX_jX_kX_{\ell}.
\end{equation}

\noindent
\textbf{Remarks:}
\begin{enumerate}
	\item In this construction, $w(D)$ depends on the position of the cut,  
	however, it always takes the value in $Z(\mathcal{U}(L)).$ 
	Furthermore, for many diagrams, $ w(D)$ is independent of the position of the cut, since it can be computed from  lower order diagrams by using the relations given in \S \ref{SEC:A1WSframe}. It is obvious from symmetry that $w(D)$  is independent of the position of the cut for low order diagrams.
	\item If $ D$ has only one chord, then $w(D)$ is the Casimir:
	\begin{equation}
		w\left( 
		\begin{tikzpicture}[baseline=0]
			\draw[very thick] (0,0) circle (0.6); 
			\draw[thick] ({0.6*cos(90)},{0.6*sin(90)}) coordinate (A) -- ({0.6*cos(270)},{0.6*sin(270)}) coordinate(B);
			\fill (A) circle(1.5pt); \fill (B) circle(1.5pt); 
		\end{tikzpicture}
		\right) 
		= 
		\sum_{i,j} C^{ij} X_i X_j = c.
	\end{equation}  
	\item If $ D $ has an isolated chord, it is factored out:
	\begin{equation}
		w\left( 
		\begin{tikzpicture}[baseline=0]
			\draw[dotted,thick] (0,0) circle (0.7); 
			\draw[very thick] (130:0.7) arc (130:230:0.7);
			\draw[thick] ({0.7*cos(140)},{0.7*sin(140)}) coordinate (A) to[bend left] ({0.7*cos(220)},{0.7*sin(220)}) coordinate(B);
			\fill (A) circle(1.5pt); \fill (B) circle(1.5pt); 
		\end{tikzpicture}
		\right) 
		= 
		c\cdot w \left(
		\begin{tikzpicture}[baseline=0]
			\draw[dotted,thick] (0,0) circle (0.7); 
			\draw[very thick] (130:0.7) arc (130:230:0.7);
		\end{tikzpicture}
		\right). \label{isofac}
	\end{equation}
	
\end{enumerate}

%
\section{Color Lie algebras and $\Z2$-graded Lie algebra $A1_{\epsilon}$} \label{SEC:colorLie}
\setcounter{equation}{0}

\subsection{Color Lie algebra as $S$-Lie algebra}

We give a definition of color Lie algebras according to \cite{sch} and point out that any color Lie algebra is an $S$-Lie algebra. 

Let $\g$ be a vector space over $\mathbb{K}.$ 
We assume that $\g$ is a direct sum of subspaces each of them is labeled by an element of an abelian group $\Gamma$:
\begin{equation}
	\g = \bigoplus_{\gamma \in \Gamma} \g_{\gamma}. \label{gradedVS}
\end{equation}
\begin{DEF}
	A commuting factor $\epsilon$ on $\Gamma$ is a mapping
	\begin{equation}
		\epsilon : \Gamma \times \Gamma \to \mathbb{K}_{\ast}
	\end{equation}
	such that
	\begin{align}
		\epsilon(\alpha,\beta) \epsilon(\beta,\alpha) &= 1,
		\\
		\epsilon(\alpha, \beta+\gamma) &= \epsilon(\alpha,\beta) \epsilon(\alpha,\gamma),
		\\
		\epsilon(\alpha+\beta, \gamma) &= \epsilon(\alpha,\gamma) \epsilon(\beta,\gamma)
	\end{align}
	for $\alpha, \beta, \gamma \in \Gamma.$ $\mathbb{K}_{\ast}$ denotes the multiplicative group of nonzero elements of $\mathbb{K}.$ 
\end{DEF}

\begin{DEF} \label{DEF:CLA}
	The $\Gamma$-graded vector space in \eqref{gradedVS} is referred to as a color Lie algebra if it admits an bilinear mapping (the graded Lie bracket) $ \llbracket \ , \ \rrbracket : \g \times \g \to \g $ satisfying the following identities 
	\begin{align}
	&	\llbracket X , Y \rrbracket \in \g_{\alpha+\beta},
		\\
	&	\llbracket X , Y \rrbracket = -\epsilon(\alpha,\beta) \llbracket Y, X \rrbracket,
		\\
	&	\epsilon(\gamma,\alpha) \llbracket X, \llbracket Y, Z \rrbracket \rrbracket + \mathrm{cyclic} = 0
	\end{align}
	for $ X \in \g_{\alpha}, Y \in \g_{\beta}, Z \in \g_{\gamma} $ and $\alpha, \beta, \gamma \in \Gamma.$
\end{DEF}

More generally, an algebras $R$ is called $\Gamma$-graded if its underlying vector space is $\Gamma$-graded:
\begin{equation}
	R = \bigoplus_{\gamma \in \Gamma} R_{\gamma}.
\end{equation}
and if, furthermore
\begin{equation}
	R_{\alpha} R_{\beta} \subset R_{\alpha+\beta}, \quad \forall \alpha, \beta \in \Gamma
\end{equation}
The graded Lie bracket is realized in $R$ by 
\begin{equation}
	\llbracket X, Y \rrbracket = XY - \epsilon(\alpha,\beta) YX, \quad X \in R_{\alpha}, Y \in R_{\beta} 
\end{equation}

The simplest example of the color Lie algebras is the trivial $\Gamma$ having only one element. In this case, the color Lie algebra $\g$ is identical to an ordinary Lie algebra and the graded Lie bracket is realized by the commutator. 
The first non-trivial case  will be $ \Gamma = \mathbb{Z}_2 = \{ \, 0,\, 1\,  \}$. 
If we take the commuting factor $ \epsilon(\alpha,\beta) = (-1)^{\alpha\beta}$, then Definition \ref{DEF:CLA} is reduced to that of the Lie superalgebras and the graded Lie bracket is realized by either commutator or anticommutator depending on the product $\alpha\beta$. 
The second non-trivial example is $ \Gamma = \Z2 := \mathbb{Z}_2 \times \mathbb{Z}_2$ for which there are two possible choices of the commuting factor. For $ \alpha = (\alpha_1,\alpha_2), \beta = (\beta_1,\beta_2) \in \Z2$, let
\begin{equation}
	\alpha \cdot \beta = \alpha_1 \beta_2 - \alpha_2 \beta_1, \quad 
	\text{or} \quad \alpha_1 \beta_1 + \alpha_2 \beta_2. \label{comfac}
\end{equation}
Then, $ \epsilon(\alpha,\beta) = (-1)^{\alpha\cdot\beta}$ is a commuting factor and the color Lie algebra corresponding to the first (resp. second) choice is referred to as $\Z2$-graded Lie algebra (resp. $\Z2$-graded Lie superalgebra) in the literature. 
The  graded Lie bracket of the $\Z2$-graded Lie (super)algebras is also realized by either commutator or anticommutator depending on the parity of  $\alpha\cdot\beta$.

It is not difficult to see that the color Lie algebra is an example of the $S$-Lie algebra. 
The following is immediately verified by direct computation. 
\begin{prop} \label{PROP:ColorS}
	Let $\g$ be a color Lie algebra. Set 
	\begin{equation}
		S(X,Y) = \epsilon(\alpha,\beta) Y \otimes X, \quad 
		f(X,Y) = \llbracket X, Y \rrbracket, \quad X, Y \in \g
	\end{equation}
	then $\g$ satisfies Definition \ref{DEF:SLie}.  
\end{prop}

\subsection{Minimal $\Z2$-graded Lie algebra $A1_{\epsilon}$}

A $\Z2$-graded Lie algebra consists of four vector subspaces each of which is labeled by an element of $\Z2$:
\begin{equation}
	\g = \g_{(0,0)} \oplus \g_{(1,0)} \oplus \g_{(0,1)}  \oplus \g_{(1,1)}. 
\end{equation}
If $ \dim \g_{\bm{a}} = 1\  (\text{i.e.}, \dim \g = 4), \forall \bm{a} \in \Z2$, then it will give an \textit{minimal} non-trivial example of the $\Z2$-graded Lie algebra.  
Such a minimal $\Z2$-graded Lie (super)algebra has been classified in \cite{FZminimal}. 
We take one of them which is called $ A1_{\epsilon} $  in \cite{FZminimal} 
and denote its basis as
\begin{equation}
	H \in \g_{(0,0)}, \qquad Q_1 \in \g_{(1,0)}, \qquad Q_2 \in \g_{(0,1)}, \qquad Q_3 \in \g_{(1,1)}
\end{equation} 
which subject to the relations (we use (anti)commutator notation for the graded Lie bracket)
\begin{equation}
	[H, Q_k] = 0, \qquad \{ Q_1, Q_2\} = Q_3, \qquad \{ Q_2, Q_3\} = \epsilon Q_1, \qquad \{ Q_3, Q_1\} = Q_2
\end{equation}
with $\epsilon = \pm 1.$ 

The algebra $A1_{\epsilon}$ has a non-degenerate invariant bilinear form, so one may use it to construct a weight system. 
Invariant bilinear forms on $\Z2$-graded Lie algebras are summarized in Appendix. 
The invariant bilinear form of $A1_{\epsilon}$ is obtained from a four-dimensional reducible representation. Let $ V = V_{(0,0)} \oplus V_{(1,0)} \oplus V_{(0,1)} \oplus V_{(1,1)} $ be a four-dimensional $\Z2$-graded $A1_{\epsilon}$-module  and $ v_{00}, v_{10}, v_{01} $ and $v_{11}$ be a basis of the corresponding subspace $V_{\bm{a}}$, 
i.e., $\dim V_{\bm{a}} = 1 $ for any $ \bm{a} \in \Z2.$ 
Define the action of $A1_{\epsilon}$ on $V$ by
\begin{alignat}{3}
	H v_{00} &= \sqrt{2\epsilon} v_{00}, &\qquad
	Q_1 v_{01} &= v_{11}, &\qquad Q_1 v_{11} &= v_{01}, 
	\notag \\ 
	Q_2 v_{10} &= v_{11}, & Q_2 v_{11} &= \epsilon v_{10}, &
	Q_3 v_{10} &= v_{01}, \qquad Q_3 v_{01} = \epsilon v_{10}
\end{alignat}
and all others vanish. 
This defines a reducible representation $\rho$ of $ A1_{\epsilon}$ where $\rho(H) =  \mathrm{diag}(\sqrt{2\epsilon}, 0, 0, 0) $ and $ \rho(Q_k) $ is identical to the adjoint representation. 
Taking the $ 4 \times 4 $ identity matrix as the matrix $M$ of \eqref{MatrixM}, we define the invariant bilinear form on $A1_{\epsilon}$ by
\begin{equation}
	B_{ab} := \frac{\epsilon}{2} tr( \rho(Q_a) \rho(Q_b))
\end{equation}
where $ a, b = 0, 1, 2, 3 $ and $ Q_0 :=H $ (we often write $Q_0$ instead of $H$ in the sequel). 
Explicitly, the bilnear form is given by the diagonal matrix:
\begin{equation}
	B = \mathrm{diag}(1,\epsilon,1,1)
\end{equation} 
which shows that $ B$ is symmetric, non-degenerate and  its inverse, denoted by $C$, is identical to itself ($C^{ab} = B_{ab}$).  

From Proposition \ref{PROP:S-Casi} (see also Proposition  \ref{PROP:CasimirDef}), the non-graded Casimir element of $A1_{\epsilon}$ is given by
\begin{equation}
	c = \sum_{a,b =0}^3 C^{ab} Q_a Q_b = H^2 + \epsilon Q_1^2 + Q_2^2 + Q_3^2. \label{CasA1e}
\end{equation}

\section{$A1_{\epsilon}$ weight system for framed knots} \label{SEC:A1WSframe}
\setcounter{equation}{0}

The center $Z(\mathcal{U}(A1_{\epsilon}))$ is generated by $H$ and the Casimir $ c $ \eqref{CasA1e}. However, it turns out that a weight system constructed from $A1_{\epsilon}$ is a polynomial in $c$ and $ y:= c-H^2. $ 
Here we give the values of $w(D)$ for the chord diagrams of order two 
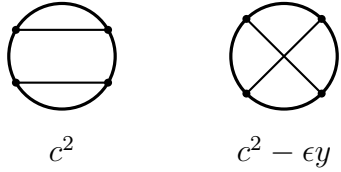
\begin{center}
\begin{tblr}{
		cell{1,2}{1,2} = {halign=c,wd=2.5cm}
	}
	\begin{tikzpicture}[baseline=0]
		\draw[thick] (30:.7) coordinate(A) -- (150:.7) coordinate(B);
		\draw[thick] (-30:.7) coordinate(C) -- (-150:.7) coordinate(D);
		\draw [very thick] (0,0) circle (.7);
		\foreach \l in {A,B,C,D}
		\fill (\l) circle (1.5pt);
	\end{tikzpicture}
	&
	\begin{tikzpicture}[baseline=0]
		\draw[thick] (45:.7) coordinate(A) -- (-135:.7) coordinate(B);
		\draw[thick] (135:.7) coordinate(C) -- (-45:.7) coordinate(D);
		\draw [very thick] (0,0) circle (.7);
		\foreach \l in {A,B,C,D}
		\fill (\l) circle (1.5pt);
	\end{tikzpicture}
	&
	\\[5pt]
	$ c^2 $  & $c^2 -\epsilon y $ 
\end{tblr}
\end{center}
and of order three
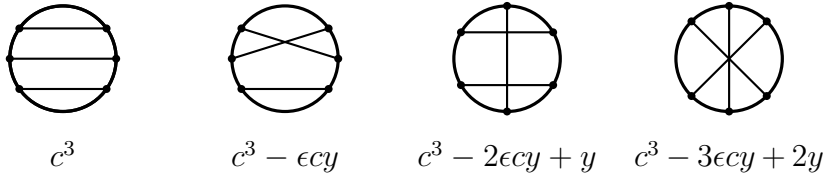
\begin{center}
	\begin{tblr}{
			cell{1,2}{1,2,3,4} = {halign=c,wd=2.5cm}
		}
		\begin{tikzpicture}[baseline=0]
			\draw[thick] (35:.7) coordinate(A) -- (145:.7) coordinate(B);
			\draw[thick] (0:.7) coordinate(C) -- (180:.7) coordinate(D);
			\draw[thick] (215:.7) coordinate(E) -- (-35:.7) coordinate(F);
			\draw [very thick] (0,0) circle (.7);
			\draw [very thick] (0,0) circle (.7);
			\foreach \l in {A,B,C,D,E,F}
			\fill (\l) circle (1.5pt);
		\end{tikzpicture}
		&
		\begin{tikzpicture}[baseline=0]
			\draw[thick] (35:.7) coordinate(A) -- (180:.7) coordinate(B);
			\draw[thick] (0:.7) coordinate(C) -- (145:.7) coordinate(D);
			\draw[thick] (215:.7) coordinate(E) -- (-35:.7) coordinate(F);
			\draw [very thick] (0,0) circle (.7);
			\foreach \l in {A,B,C,D,E,F}
			\fill (\l) circle (1.5pt);
		\end{tikzpicture}
		&
		\begin{tikzpicture}[baseline=0]
			\draw[thick] (30:.7) coordinate(A) -- (150:.7) coordinate(B);
			\draw[thick] (-30:.7) coordinate(C) -- (-150:.7) coordinate(D);
			\draw[thick] (90:.7)  coordinate(E) -- (270:.7) coordinate (F);
			\draw [very thick] (0,0) circle (.7);
			\foreach \l in {A,B,C,D,E,F}
			\fill (\l) circle (1.5pt);
		\end{tikzpicture}
		&
		\begin{tikzpicture}[baseline=0]
			\draw[thick] (45:.7) coordinate(A) -- (-135:.7) coordinate(B);
			\draw[thick] (135:.7) coordinate(C) -- (-45:.7) coordinate(D);
			\draw[thick] (90:.7) coordinate(E) -- (270:.7) coordinate(F);
			\draw [very thick] (0,0) circle (.7);
			\foreach \l in {A,B,C,D,E,F}
			\fill (\l) circle (1.5pt);
		\end{tikzpicture}
		\\[5pt]
		$ c^3 $ & $c^3-\epsilon cy$ & $c^3-2\epsilon cy+y$ & $c^3-3\epsilon cy+2y$  	
	\end{tblr}
\end{center}
The  $w(D)$ for a Jacobi diagram is obtained by resolving it to chord diagrams, e.g., 
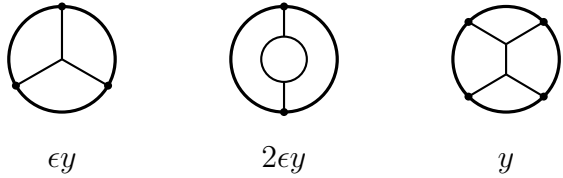
\begin{center}
	\begin{tblr}{cell{1,2}{1,2,3} = {halign=c,wd=2.5cm},
		row{2} = {mode=dmath}}
		\begin{tikzpicture}[baseline=0]
			\draw [very thick] (0,0) circle (.7);
			\draw[thick] (90:.7) coordinate(A) --(0,0);
			\draw[thick] (210:.7) coordinate(B) --(0,0);
			\draw[thick] (330:.7) coordinate(C) --(0,0);
			\foreach \l in {A,B,C}
				\fill (\l) circle (1.5pt);
		\end{tikzpicture}
		&
		\begin{tikzpicture}[baseline=0]
			\draw [very thick] (0,0) circle (.7);
			\draw[thick] (0,0) circle (.3);
			\draw[thick] (90:.7) coordinate(A) --(90:.3);
			\draw[thick] (270:.7) coordinate(B) --(270:.3);
			\foreach \l in {A,B}
			   \fill (\l) circle (1.5pt);
		\end{tikzpicture}
		&
		\begin{tikzpicture}[baseline=0]
			\draw [very thick] (0,0) circle (.7);
			\draw[thick] (45:.7) coordinate(A) --(90:.2);
			\draw[thick] (135:.7) coordinate(B) --(90:.2);
			\draw[thick] (-45:.7) coordinate(C) --(-90:.2);
			\draw[thick] (-135:.7) coordinate(D) --(-90:.2);
			\draw[thick] (90:.2)--(-90:.2);
			\foreach \l in {A,B,C,D}
			\fill (\l) circle (1.5pt);
		\end{tikzpicture}
		\\[5pt]
		\epsilon y & 2 \epsilon y & y 
	\end{tblr}
\end{center}
In principle, it is possible to calculate $w(D)$ for any diagrams once we know its value for all chord diagrams. However, it is impossible to obtain general expression for $w(D)$ for all chord diagrams.  Therefore, we derive some relations which are useful for evaluating $w(D)$ for a given diagram. Our main result is a recurrence relation for chord diagrams.
\begin{thm} \label{THM:recrel}
Let $D$ be a chord diagram and  ``a" a fixed chord in $D$. Suppose that there are $k$ chords intersecting the chord $a$ and we denote the intersecting chords by $ b_1, b_2, \dots, b_k$. 
Then $w(D)$ is a polynomial in $c$ and $y$ satisfying the recurrence relation
	\begin{align}
		w(D) &= (c-\epsilon k) w(D_a) + \epsilon (c-y) \sum_i w(D_{a,i}) 
		+ \epsilon \sum_{i<j} \big( w(D_{a,ij}^{\parallel}) -w(D_{a,ij}^{\times}) \big)
		\notag \\
		&+ \epsilon (c-y) \sum_{i< j} \big( w(D_{a,ij}^{\ell r}) + w(D_{a,ij}^{r \ell}) - w(D_{a,ij}^{\ell \ell}) - w(D_{a,ij}^{rr} ) \big)  \label{RecRel}
	\end{align}
	where $ D_a, D_{a,i} $ and $ D_{a,ij}$ are the diagrams $ D-a, D-a-b_i$ and $ D - a- b_i - b_j,$  respectively. 
	The diagrams $ D_{a,ij}^{\parallel}$ and $ D_{a,ij}^{\times } $  (resp. $D_{a,ij}^{\ell r}, D_{a,ij}^{r\ell}, D_{a,ij}^{\ell\ell}$ and $D_{a,ij}^{rr}$ ) are the diagrams obtained by adding to $ D_{a,ij}$ two (resp. one) new chords in the following way: Draw the diagram $D$ so that the chord $a$ is vertical, then add the new chords by connecting the end points which are indicated below: 
\end{thm}
\begin{center}
\begin{tblr}{Q[c,m] Q[l,m]}
diagram & \SetCell[c=1]{c} new chords
\\ \hline 
$ D_{a,ij}^{\parallel}$ & {left end of $b_i$ and  left end of $b_j$, \\
	                        right end of $b_i$ and right end of $b_j$}
\\[7pt]
$ D_{a,ij}^{\times}$ & {left end of $b_i$ and  right end of $b_j$, \\
						right end of $b_i$ and  left end of $b_j$}
\\[7pt]
	$ D_{a,ij}^{\ell r}$
& left end of $b_i$ and right end of $b_j$	
\\[7pt]
$ D_{a,ij}^{r\ell}$
& right end of $b_i$ and left end of $b_j$
\\[7pt]
$ D_{a,ij}^{\ell\ell}$
& left ends of $b_i$ and  $b_j$
\\[7pt]
$ D_{a,ij}^{r r}$
& right ends of $b_i$ and  $b_j$				
\end{tblr}
\end{center}
The diagrams appearing in \eqref{RecRel} are as follows:
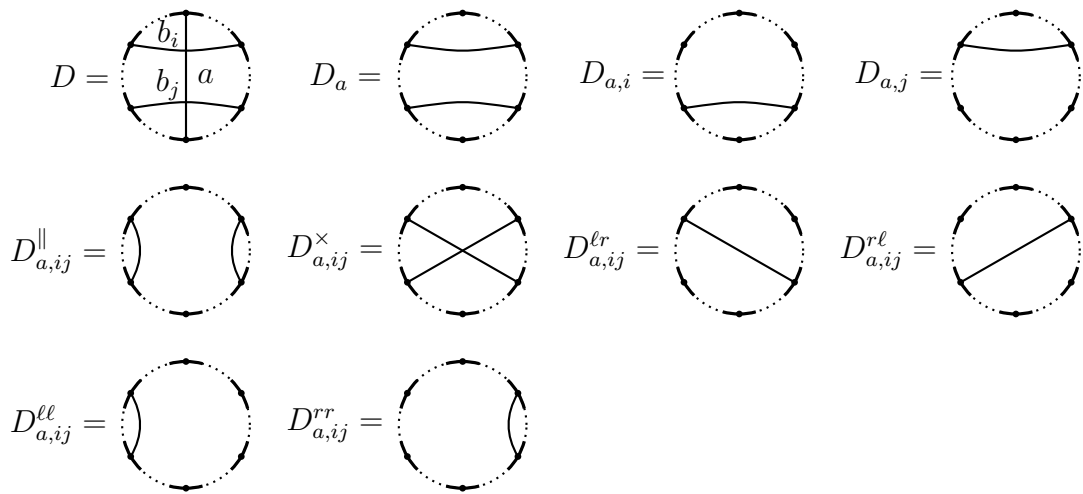
\begin{center}
\begin{tikzpicture}[scale=0.7]
	\draw[dotted,thick] (0,0) circle [radius=1.2];
	\foreach \t/\label in {30/A,90/B,150/C,210/D,270/E,330/F}{
		\draw[very thick] (\t-15:1.2) arc (\t-15:\t+15:1.2);
		\coordinate (\label) at ({1.2*cos(\t)},{1.2*sin(\t)});
		\fill (\label) circle (1.8pt);
	}
	\draw[thick] (B)--(E) node[pos=0.5,right] {$a$};
	\draw[thick,] (A) ..controls ($(A)!0.5!(C)+(0,-0.15)$).. (C) node[pos=0.7,yshift=6pt] {$b_i$};
	\draw[thick,bend left=10] (D)..controls ($(D)!0.5!(F)+(0,0.15)$).. (F) node[pos=0.3,yshift=7pt] {$b_j$};
	\node at (-2,0) {$D = $};
	\begin{scope}[xshift=5.2cm]
		\draw[dotted,thick] (0,0) circle [radius=1.2];
		\foreach \t/\label in {30/A,90/B,150/C,210/D,270/E,330/F}{
		   \draw[very thick] (\t-15:1.2) arc (\t-15:\t+15:1.2);
		   \coordinate (\label) at ({1.2*cos(\t)},{1.2*sin(\t)});
		   \fill (\label) circle (1.8pt);
		}
		\draw[thick,] (A) ..controls ($(A)!0.5!(C)+(0,-0.15)$).. (C);
		\draw[thick,bend left=10] (D)..controls ($(D)!0.5!(F)+(0,0.15)$).. (F); 
		\node at (-2.2,0) {$D_a = $};
	\end{scope}
	\begin{scope}[xshift=10.4cm]
		\draw[dotted,thick] (0,0) circle [radius=1.2];
		\foreach \t/\label in {30/A,90/B,150/C,210/D,270/E,330/F}{
			\draw[very thick] (\t-15:1.2) arc (\t-15:\t+15:1.2);
			\coordinate (\label) at ({1.2*cos(\t)},{1.2*sin(\t)});
			\fill (\label) circle (1.8pt);
		}
		\draw[thick] (D)..controls ($(D)!0.5!(F)+(0,0.15)$).. (F); 		
		\node at (-2.2,0) {$D_{a,i} = $};
	\end{scope}	
	\begin{scope}[xshift=15.6cm]
		\draw[dotted,thick] (0,0) circle [radius=1.2];
		\foreach \t/\label in {30/A,90/B,150/C,210/D,270/E,330/F}{
			\draw[very thick] (\t-15:1.2) arc (\t-15:\t+15:1.2);
			\coordinate (\label) at ({1.2*cos(\t)},{1.2*sin(\t)});
			\fill (\label) circle (1.8pt);
		}
		\draw[thick,] (A) ..controls ($(A)!0.5!(C)+(0,-0.15)$).. (C);
		\node at (-2.2,0) {$D_{a,j} = $};
	\end{scope}
	\begin{scope}[yshift=-3.3cm]
		\draw[dotted,thick] (0,0) circle [radius=1.2];
		\foreach \t/\label in {30/A,90/B,150/C,210/D,270/E,330/F}{
			\draw[very thick] (\t-15:1.2) arc (\t-15:\t+15:1.2);
			\coordinate (\label) at ({1.2*cos(\t)},{1.2*sin(\t)});
			\fill (\label) circle (1.8pt);
		}
		\draw[thick,bend left] (C) to (D);	
		\draw[thick,bend left] (F) to (A);
		\node at (-2.4,0) {$D_{a,ij}^{\parallel} = $};
	\end{scope}
	\begin{scope}[xshift=5.2cm, yshift=-3.3cm]
		\draw[dotted,thick] (0,0) circle [radius=1.2];
		\foreach \t/\label in {30/A,90/B,150/C,210/D,270/E,330/F}{
			\draw[very thick] (\t-15:1.2) arc (\t-15:\t+15:1.2);
			\coordinate (\label) at ({1.2*cos(\t)},{1.2*sin(\t)});
			\fill (\label) circle (1.8pt);
		}
		\draw[thick] (C) to (F);	
		\draw[thick] (D) to (A);
		\node at (-2.4,0) {$D_{a,ij}^{\times} = $};
	\end{scope}
	\begin{scope}[xshift=10.4cm, yshift=-3.3cm]
		\draw[dotted,thick] (0,0) circle [radius=1.2];
		\foreach \t/\label in {30/A,90/B,150/C,210/D,270/E,330/F}{
			\draw[very thick] (\t-15:1.2) arc (\t-15:\t+15:1.2);
			\coordinate (\label) at ({1.2*cos(\t)},{1.2*sin(\t)});
			\fill (\label) circle (1.8pt);
		}
		\draw[thick] (C) -- (F);
		\node at (-2.4,0) {$D_{a,ij}^{\ell r} = $};
	\end{scope}
	\begin{scope}[xshift=15.6cm, yshift=-3.3cm]
		\draw[dotted,thick] (0,0) circle [radius=1.2];
		\foreach \t/\label in {30/A,90/B,150/C,210/D,270/E,330/F}{
			\draw[very thick] (\t-15:1.2) arc (\t-15:\t+15:1.2);
			\coordinate (\label) at ({1.2*cos(\t)},{1.2*sin(\t)});
			\fill (\label) circle (1.8pt);
		}
		\draw[thick] (D) -- (A);
		\node at (-2.4,0) {$D_{a,ij}^{r\ell} = $};
	\end{scope}
	\begin{scope}[yshift=-6.6cm]
		\draw[dotted,thick] (0,0) circle [radius=1.2];
		\foreach \t/\label in {30/A,90/B,150/C,210/D,270/E,330/F}{
			\draw[very thick] (\t-15:1.2) arc (\t-15:\t+15:1.2);
			\coordinate (\label) at ({1.2*cos(\t)},{1.2*sin(\t)});
			\fill (\label) circle (1.8pt);
		}
		\draw[thick,bend left] (C) to (D);	
		\node at (-2.4,0) {$D_{a,ij}^{\ell\ell} = $};
	\end{scope}
	\begin{scope}[xshift=5.2cm, yshift=-6.6cm]
		\draw[dotted,thick] (0,0) circle [radius=1.2];
		\foreach \t/\label in {30/A,90/B,150/C,210/D,270/E,330/F}{
			\draw[very thick] (\t-15:1.2) arc (\t-15:\t+15:1.2);
			\coordinate (\label) at ({1.2*cos(\t)},{1.2*sin(\t)});
			\fill (\label) circle (1.8pt);
		}
		\draw[thick,bend left] (F) to (A);
		\node at (-2.4,0) {$D_{a,ij}^{rr} = $};
	\end{scope}
\end{tikzpicture}
\end{center}

The recurrence relation \eqref{RecRel} shows that the $A1_{\epsilon}$ weight system is a hybrid of those of $sl(2) $ and $gl(1|1)$. 
The first four terms of $\eqref{RecRel}$ (from $w(D_a)$ to $ w(D_{a,ij}^{\times})$) are also found in the $sl(2)$ recurrence relation and the last four terms (from $ w(D_{a,ij}^{\ell r})$ to $ w(D_{a,ij}^{rr})$) are in that of $ gl(1|1)$. 

We also present some relations involving Jacobi diagrams. 
Using the STU and AS relations, it is easy to see that
\begin{equation}
	w\left(
	\begin{tikzpicture}[baseline=0,scale=0.8]
		\draw[very thick] (0,0) circle (1);
		\draw (50:1) coordinate(A) -- (40:0.7);
		\draw (130:1) coordinate(B)-- (140:0.7);
		\node at (90:0.76) {$\cdots\cdot$};
		\fill[pattern=north east lines] (0,0.2) ellipse (.8 and .4);
		\draw[thick] (0,-0.55) circle (0.15);
		\draw[thick] (0,-0.2)--(0,-0.4);
		\draw[thick] (0,-0.68)--(0,-1) coordinate (C);
		\foreach \l in {A,B,C}
			\fill (\l) circle (1.8pt);
	\end{tikzpicture}
	\right)
	= 2
	w\left(
	 \begin{tikzpicture}[baseline=0,scale=0.8]
	 	\draw[very thick] (0,0) circle (1);
	 	\draw (50:1) coordinate(A) -- (40:0.7);
	 	\draw (130:1) coordinate(B)-- (140:0.7);
	 	\node at (90:0.76) {$\cdots\cdot$};
	 	\fill[pattern=north east lines] (0,0.2) ellipse (.8 and .4);
	 	\draw[thick] (0,-0.2)--(0,-0.6);
	 	\draw[thick] (0,-0.6)--(-70:1) coordinate(C);
	 	\draw[thick] (0,-0.6)--(-110:1) coordinate(D);
	 	\foreach \l in {A,B,C,D}
	 		\fill (\l) circle (1.8pt);
	\end{tikzpicture}
	\right)
\end{equation}
where the shaded region indicates all other edges. 
This is true for any weight systems because we need the STU and AS relations to show this.  
The relations specific to $A1_{\epsilon}$ are given in the following propositions: 

\begin{prop}
\label{PROP:Y}
\begin{equation}
	w\left(
	  \begin{tikzpicture}[baseline=0,scale=0.8]
	  	\draw[dotted,thick] (0,0) circle [radius=1];
	  	\draw[very thick] (70:1) arc (70:110:1);
	  	\draw[very thick] (240:1) arc (240:300:1);
	  	\draw[thick] (90:1) coordinate(A) --(270:.5) coordinate(v);
	  	\draw[thick] (v) -- (250:1) coordinate (B);
	  	\draw[thick] (v) -- (290:1) coordinate (C);
	  	\foreach \l in {A,B,C}
	  		\fill (\l) circle (1.8pt);
	  \end{tikzpicture}
	\right)
	= \epsilon \cdot 
	  w\left(
	    \begin{tikzpicture}[baseline=0,scale=0.8]
	    	\draw[dotted,thick] (0,0) circle [radius=1];
	    	\draw[very thick] (70:1) arc (70:110:1);
	    	\draw[very thick] (250:1) arc (250:290:1);
	    	\draw[thick] (90:1) coordinate(A) -- (270:1) coordinate(B);
	    	\foreach \l in {A,B}
	    		\fill (\l) circle (1.8pt);
	    \end{tikzpicture}
	  \right)
	  - \epsilon (c-y)
	  w\left(
	  	\begin{tikzpicture}[baseline=0,scale=0.8]
	  		\draw[dotted,thick] (0,0) circle [radius=1];
	  		\draw[very thick] (70:1) arc (70:110:1);
	  		\draw[very thick] (250:1) arc (250:290:1);
	  	\end{tikzpicture}
	  \right). \label{y}
\end{equation}
\end{prop}

\begin{prop} \label{PROP:geta}
	\begin{align}
		w\left(
			\begin{tikzpicture}[baseline=0,scale=.8]
				\draw[dotted,thick] (0,0) circle [radius=1];
				\foreach \t/\l in {45/A,135/B,225/C,315/D}{
					\fill (\t:1) circle (1.8pt) coordinate (\l);
					\draw[very thick] (\t-15:1) arc(\t-15:\t+15:1);
				}
				\coordinate (v1) at (-0.4,0); \coordinate (v2) at (0.4,0);
				\draw[thick] (v1)--(v2);
				\draw[thick] (A)--(v2); \draw[thick] (D)--(v2);
				\draw[thick] (B)--(v1); \draw[thick] (C)--(v1);
			\end{tikzpicture}
		\right) 
		&=
		\epsilon\cdot 
		w\left(
			\begin{tikzpicture}[baseline=0,scale=.8]
				\draw[dotted,thick] (0,0) circle [radius=1];
				\foreach \t/\l in {45/A,135/B,225/C,315/D}{
					\fill (\t:1) circle (1.8pt) coordinate (\l);
					\draw[very thick] (\t-15:1) arc(\t-15:\t+15:1);
				}
				\draw[thick] (A)to[bend left] (B); \draw[thick] (C) to[bend left] (D);
			\end{tikzpicture}
		-
			\begin{tikzpicture}[baseline=0,scale=.8]
				\draw[dotted,thick] (0,0) circle [radius=1];
				\foreach \t/\l in {45/A,135/B,225/C,315/D}{
					\fill (\t:1) circle (1.8pt) coordinate (\l);
					\draw[very thick] (\t-15:1) arc(\t-15:\t+15:1);
				}
				\draw[thick] (A)--(C); \draw[thick] (B)--(D);
			\end{tikzpicture}
		\right)
		\notag \\
		&+ \epsilon (c-y) \cdot w
		\left(
			\begin{tikzpicture}[baseline=0,scale=.8]
				\draw[dotted,thick] (0,0) circle [radius=1];
				\foreach \t/\l in {45/A,135/B,225/C,315/D}{
					\coordinate (\l) at (\t:1);
					\draw[very thick] (\t-15:1) arc(\t-15:\t+15:1);
				}
				\draw[thick] (C)--(A); 
				\fill (A) circle (1.8pt); \fill (C) circle (1.8pt);
			\end{tikzpicture}
			+
			\begin{tikzpicture}[baseline=0,scale=.8]
				\draw[dotted,thick] (0,0) circle [radius=1];
				\foreach \t/\l in {45/A,135/B,225/C,315/D}{
					\coordinate (\l) at (\t:1);
					\draw[very thick] (\t-15:1) arc(\t-15:\t+15:1);
				}
				\draw[thick] (B)--(D); 
				\fill (B) circle (1.8pt); \fill (D) circle (1.8pt);
			\end{tikzpicture}
			-
			\begin{tikzpicture}[baseline=0,scale=.8]
				\draw[dotted,thick] (0,0) circle [radius=1];
				\foreach \t/\l in {45/A,135/B,225/C,315/D}{
					\coordinate (\l) at (\t:1);
					\draw[very thick] (\t-15:1) arc(\t-15:\t+15:1);
				}
				\draw[thick] (A) to[bend left] (B); 
				\fill (A) circle (1.8pt); \fill (B) circle (1.8pt);
			\end{tikzpicture}
			-
			\begin{tikzpicture}[baseline=0,scale=.8]
				\draw[dotted,thick] (0,0) circle [radius=1];
				\foreach \t/\l in {45/A,135/B,225/C,315/D}{
					\coordinate (\l) at (\t:1);
					\draw[very thick] (\t-15:1) arc(\t-15:\t+15:1);
				}
				\draw[thick] (C) to[bend left] (D); 
				\fill (D) circle (1.8pt); \fill (C) circle (1.8pt);
			\end{tikzpicture}
		\right).
		\label{Geta}
	\end{align}
\end{prop}

\begin{prop} \label{PROP:cross}
	\begin{align}
		w\left(
		\begin{tikzpicture}[baseline=0,scale=.8]
			\draw[dotted,thick] (0,0) circle [radius=1];
			\foreach \t/\l in {45/A,135/B,225/C,315/D}{
				\fill (\t:1) circle (1.8pt) coordinate (\l);
				\draw[very thick] (\t-15:1) arc(\t-15:\t+15:1);
			}
			\draw[thick] (A)--(C);
			\draw[thick] (B)--(D); 
			\draw[thick] ($(A)!0.75!(C)$)--($(B)!0.25!(D)$);
		\end{tikzpicture}
		\right) 
		&=
		\epsilon\cdot 
		w\left(
		\begin{tikzpicture}[baseline=0,scale=.8]
			\draw[dotted,thick] (0,0) circle [radius=1];
			\foreach \t/\l in {45/A,135/B,225/C,315/D}{
				\fill (\t:1) circle (1.8pt) coordinate (\l);
				\draw[very thick] (\t-15:1) arc(\t-15:\t+15:1);
			}
			\draw[thick] (A)to[bend right] (D); \draw[thick] (B) to[bend left] (C);
		\end{tikzpicture}
		-
		\begin{tikzpicture}[baseline=0,scale=.8]
			\draw[dotted,thick] (0,0) circle [radius=1];
			\foreach \t/\l in {45/A,135/B,225/C,315/D}{
				\fill (\t:1) circle (1.8pt) coordinate (\l);
				\draw[very thick] (\t-15:1) arc(\t-15:\t+15:1);
			}
			\draw[thick] (A) to[bend left] (B); \draw[thick] (C) to[bend left] (D);
		\end{tikzpicture}
		\right)
		\notag \\
		&+ \epsilon (c-y) \cdot w
		\left(
		\begin{tikzpicture}[baseline=0,scale=.8]
			\draw[dotted,thick] (0,0) circle [radius=1];
			\foreach \t/\l in {45/A,135/B,225/C,315/D}{
				\coordinate (\l) at (\t:1);
				\draw[very thick] (\t-15:1) arc(\t-15:\t+15:1);
			}
			\draw[thick] (C) to[bend left] (D); 
			\fill (D) circle (1.8pt); \fill (C) circle (1.8pt);
		\end{tikzpicture}
		+
		\begin{tikzpicture}[baseline=0,scale=.8]
			\draw[dotted,thick] (0,0) circle [radius=1];
			\foreach \t/\l in {45/A,135/B,225/C,315/D}{
				\coordinate (\l) at (\t:1);
				\draw[very thick] (\t-15:1) arc(\t-15:\t+15:1);
			}
			\draw[thick] (A) to[bend left] (B); 
			\fill (A) circle (1.8pt); \fill (B) circle (1.8pt);
		\end{tikzpicture}
		-
		\begin{tikzpicture}[baseline=0,scale=.8]
			\draw[dotted,thick] (0,0) circle [radius=1];
			\foreach \t/\l in {45/A,135/B,225/C,315/D}{
				\coordinate (\l) at (\t:1);
				\draw[very thick] (\t-15:1) arc(\t-15:\t+15:1);
			}
			\draw[thick] (B) to[bend left] (C); 
			\fill (B) circle (1.8pt); \fill (C) circle (1.8pt);
		\end{tikzpicture}
		-
		\begin{tikzpicture}[baseline=0,scale=.8]
			\draw[dotted,thick] (0,0) circle [radius=1];
			\foreach \t/\l in {45/A,135/B,225/C,315/D}{
				\coordinate (\l) at (\t:1);
				\draw[very thick] (\t-15:1) arc(\t-15:\t+15:1);
			}
			\draw[thick] (D) to[bend left] (A); 
			\fill (D) circle (1.8pt); \fill (A) circle (1.8pt);
		\end{tikzpicture}
		\right).
		\label{cross}
	\end{align}
\end{prop}
Proposition \ref{PROP:Y} and \ref{PROP:geta} also show the hybrid nature of $A1_{\epsilon}$ weight systems. The relation \eqref{y} for $sl(2)$ (resp. $gl(1|1)$) has only the first (resp. second) term. The relation \eqref{Geta} for $ sl(2) $ (resp. $gl(1|1)$) has only the terms with two (resp. one) chords. 
These propositions and Theorem \ref{THM:recrel} are proved in the next section, where we will see that Proposition \ref{PROP:geta} and \ref{PROP:cross} play a crucial role. 
In the rest of this section, some results from these relations are given. 

\begin{enumerate}
	\item The formula of $w(D_n)$ for the following chord diagram:
	\begin{equation}
		D_n = 
	\begin{tikzpicture}[baseline=-3pt, scale=0.8]
		\draw[very thick] (0,0) circle (1);
		\foreach \t in {20,40,220}{
			\draw[thick] (\t:1) -- (180-\t:1);
			\fill (\t:1) circle (1.8pt);
			\fill (180-\t:1) circle (1.8pt);
		}
		\draw[thick] (90:1) coordinate(A) -- (270:1) coordinate(B);
		\foreach \l in {A,B}
			\fill (\l) circle (1.8pt);
		\node at (.5,0){$\vdots$};
		\node at (1.4,0){$\Biggr\} n $};
	\end{tikzpicture}
	\qquad 
	w(D_n)=c^{n+1} - \big( c^n-(c-\epsilon)^n \big) y.
	\end{equation}
	This is proved by induction on the number of chords using Proposition \ref{PROP:Y}. 
\item The formula of the $A1_{\epsilon}$ universal weight system for the following Jacobi diagrams:
\begin{equation}
	\begin{tikzpicture}[baseline=0,scale=.8]
		\tikzset{
			position label/.style={
				above = 1pt,
				text height = 1.5ex,
				text depth = 1ex
			},
			brace/.style={
				decoration={brace},
				decorate
			}
		}
		\draw[very thick] circle [radius=1];
		\foreach \t/\l in {0/A,45/B,100/C,135/D,180/E}
			\fill (\t:1) coordinate (\l) circle(1.8pt);
		\draw[thick] (A)--(E);
		\foreach \l in {B,C,D}
			\draw[thick] (\l) -- ($(A)!(\l)!(E)$);
		\node[left, below,xshift=-8pt] at (B) {$\cdots$};
		\draw[brace,thick] ($(D)+(-.1,0.4)$) -- ($(B)+(.1,0.4)$) node [position label, pos=0.5] {$n$};
		\node at ($(E)+(-.9,0)$) {$D_n = $};
		\node at (0,-2) {$w(D_n) = \epsilon^n y$};
		\begin{scope}[xshift=6.5cm]
			\draw[very thick] circle [radius=1];
			\foreach \t/\l in {0/A,65/B,110/C,135/D,180/E,30/P1,-30/P2}
			\fill (\t:1) coordinate (\l) circle(1.8pt);
			\draw[thick] (A)--(E);
			\foreach \l in {B,C,D}
			\draw[thick] (\l) -- ($(A)!(\l)!(E)$);
			\node[left, below,xshift=-8pt,yshift=-3pt] at (B) {$\cdots$};
			\draw[brace,thick] ($(D)+(-.1,0.4)$) -- ($(B)+(.1,0.25)$) node [position label, pos=0.5] {$n$};
			\node at ($(E)+(-.9,0)$) {$D'_n = $};
			\draw[thick] (P1) to[bend right] (P2);
			\node at (0,-2) {$w(D'_n) = \epsilon^n (c-\epsilon) y$};
		\end{scope}
	\end{tikzpicture} 
	\label{teeth}
\end{equation}
These immediately follow from the recurrence relations
\begin{align}
	w(D_n) &= c w(D_{n-1}) - w(D'_{n-1}), \label{hige1}
	\\
	w(D'_n) &= (\epsilon-c) w(D'_{n-1}) - c(\epsilon-c) w(D_{n-1}). \label{hige2}
\end{align}
\eqref{hige1} is verified by applying the STU relation to  the right vertical edge.   
Repeating the same and then applying Proposition \ref{PROP:Y}, one may prove \eqref{hige2}.
\item Some more examples of Jacobi diagrams:
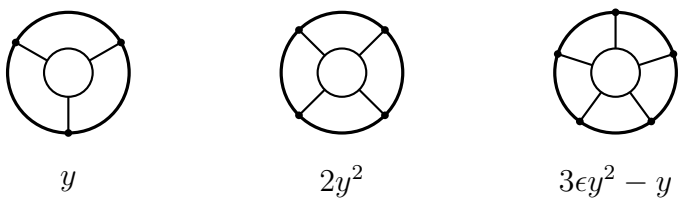
\begin{center}
	\begin{tikzpicture}[baseline=0,scale=.8]
		\draw [very thick] (0,0) circle [radius=1];
		\draw[thick] (0,0) circle [radius=0.4];
		\foreach \t in {30, 150, 270}{
			\fill (\t:1)  circle (1.8pt);
			\draw[thick] (\t:1) -- (\t:0.4);
		}
		\node at(0,-1.8) {$y$};	
	\begin{scope}[xshift=4.5cm]
			\draw [very thick] (0,0) circle [radius=1];
			\draw[thick] (0,0) circle [radius=0.4];
			\foreach \t in {45,135,225,315}{
				\fill (\t:1)  circle (1.8pt);
				\draw[thick] (\t:1) -- (\t:0.4);
			}	
			\node at(0,-1.8) {$2y^2$};	
	\end{scope}
	\begin{scope}[xshift=9cm]
		\draw [very thick] (0,0) circle [radius=1];
		\draw[thick] (0,0) circle [radius=0.4];
		\foreach \t in {18,90,162,234,306}{
			\fill (\t:1)  circle (1.8pt);
			\draw[thick] (\t:1) -- (\t:0.4);
		}	
		\node at(0,-1.8) {$3\epsilon y^2-y$};	
	\end{scope}
	\end{tikzpicture}
\end{center}
We illustrate the calculation of the diagram in the middle. By the STU relation
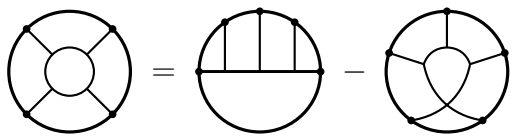
\[
	\begin{tikzpicture}[baseline=-3pt,scale=.8]
		\draw [very thick] (0,0) circle [radius=1];
		\draw[thick] (0,0) circle [radius=0.4];
		\foreach \t in {45,135,225,315}{
			\fill (\t:1)  circle (1.8pt);
			\draw[thick] (\t:1) -- (\t:0.4);
		}	
	\end{tikzpicture} 
	\ = \  
	\begin{tikzpicture}[baseline=-3pt,scale=.8]
		\draw [very thick] (0,0) circle [radius=1];
		\foreach \t/\l in {0/A,55/B,90/C,125/D,180/E}{
			\fill (\t:1) coordinate(\l) circle (1.8pt);
			\draw[thick] (A) -- (180:1); 
		}
		\foreach \l in {B,C,D}
			\draw[thick] (\l) -- ($(A)!(\l)!(E)$);
	\end{tikzpicture}
	\ - \ 
	\begin{tikzpicture}[baseline=-3pt,scale=.8]
		\draw [very thick] (0,0) circle [radius=1];
		\foreach \t/\l in {18/A1,90/A2,162/A3,234/A4,306/A5}{
			\fill (\t:1) coordinate(\l) circle (1.8pt);
		}	
		\foreach \t/\la in {18/B1,90/B2,162/B3}
			\draw[thick] (\t:1) -- (\t:0.4) coordinate(\la);
		\draw[thick] (A4) to[bend right=35] (B1) to[bend right=35] (B2) to[bend right=35] (B3) to[bend right=35] (A5);
	\end{tikzpicture}
\]
Once more by the STU relation, we get
\[
	\begin{tikzpicture}[baseline=-3pt,scale=.8]
		\draw [very thick] (0,0) circle [radius=1];
		\foreach \t/\l in {18/A1,90/A2,162/A3,234/A4,306/A5}{
			\fill (\t:1) coordinate(\l) circle (1.8pt);
		}	
		\foreach \t/\la in {18/B1,90/B2,162/B3}
		\draw[thick] (\t:1) -- (\t:0.4) coordinate(\la);
		\draw[thick] (A4) to[bend right=35] (B1) to[bend right=35] (B2) to[bend right=35] (B3) to[bend right=35] (A5);
	\end{tikzpicture}
	\ = \ 
	\begin{tikzpicture}[baseline=-3pt,scale=.8]
		\draw[very thick] circle [radius=1];
		\foreach \t/\l in {0/A,80/B,120/C,180/E,40/P1,-40/P2}
		\fill (\t:1) coordinate (\l) circle(1.8pt);
		\draw[thick] (A)--(E);
		\foreach \l in {B,C}
		\draw[thick] (\l) -- ($(A)!(\l)!(E)$);
		\draw[thick] (P1) to[bend right] (P2);
	\end{tikzpicture}
	\ - \ 
	\begin{tikzpicture}[baseline=-3pt,scale=.8]
		\draw[very thick] (0,0) circle [radius=1];
		\foreach \t/\l in {45/A,135/B,225/C,315/D,10/P1,170/P2}{
			\fill (\t:1) circle (1.8pt) coordinate (\l);
		}
		\coordinate (v1) at (-0.4,0); \coordinate (v2) at (0.4,0);
		\draw[thick] (v1)--(v2);
		\draw[thick] (A)--(v2); \draw[thick] (D)--(v2);
		\draw[thick] (B)--(v1); \draw[thick] (C)--(v1);
		\draw[thick] (P1) to[bend right] (P2);
	\end{tikzpicture}
\]
We obtain by Proposition \ref{PROP:geta} that 
\[
 w\left(  
 	\begin{tikzpicture}[baseline=-3pt,scale=.8]
 		\draw[very thick] (0,0) circle [radius=1];
 		\foreach \t/\l in {45/A,135/B,225/C,315/D,10/P1,170/P2}{
 			\fill (\t:1) circle (1.8pt) coordinate (\l);
 		}
 		\coordinate (v1) at (-0.4,0); \coordinate (v2) at (0.4,0);
 		\draw[thick] (v1)--(v2);
 		\draw[thick] (A)--(v2); \draw[thick] (D)--(v2);
 		\draw[thick] (B)--(v1); \draw[thick] (C)--(v1);
 		\draw[thick] (P1) to[bend right] (P2);
 	\end{tikzpicture}
 \right) 
 = cy - 2\epsilon y + 2 y^2.
\]
These, together with \eqref{teeth}, gives enough materials to calculate the value of the middle diagram. 
\item The values of $w(D)$ for \textit{indecomposable} chord diagrams of order four. 
The diagram is indecomposable if it is not a connected sum of two diagrams. 
\begin{align}
	\begin{tikzpicture}[baseline,scale=.8]
		\draw[very thick] circle [radius=1];
		\foreach \t/\l in {0/P1,45/P2,90/P3,135/P4,180/P5,225/P6,270/P7,315/P8}
		\fill (\t:1) coordinate(\l) circle (1.8pt);
		\draw[thick] (P2) to[bend right] (P8);
		\draw[thick] (P4) to[bend left] (P6);
		\draw[thick] (P1) -- (P5);
		\draw[thick] (P3) -- (P7);
	\end{tikzpicture}
	& \qquad c^4 -3c^2 \epsilon y +3cy - \epsilon y
	\\[5pt]
	\begin{tikzpicture}[baseline,scale=.8]
		\draw[very thick] circle [radius=1];
		\foreach \t/\l in {50/P1,110/P2,130/P3,160/P4,180/P5,210/P6,230/P7,280/P8}
		  \fill (\t:1) coordinate(\l) circle (1.8pt);
		  \draw[thick] (P1) to[bend left] (P3);
		  \draw[thick] (P2) to[bend left] (P5);
		  \draw[thick] (P4) to[bend left] (P7);
		  \draw[thick] (P6) to[bend left] (P8);
	\end{tikzpicture}
	& \qquad c^4 -3c^2 \epsilon y + 2cy - \epsilon y + y^2
	\\[5pt]
	\begin{tikzpicture}[baseline,scale=.8]
		\draw[very thick] circle [radius=1];
		\foreach \t/\l in {350/P1,45/P2,90/P3,135/P4,190/P5,210/P6,270/P7,330/P8}
		\fill (\t:1) coordinate(\l) circle (1.8pt);
		\draw[thick] (P6) -- (P8);
		\draw[thick] (P3) -- (P7);
		\draw[thick] (P2) -- (P5);
		\draw[thick] (P1) -- (P4);
	\end{tikzpicture}
	& \qquad c^4 -4c^2 \epsilon y + 4cy - 2\epsilon y + y^2
	\\[5pt]
	\begin{tikzpicture}[baseline,scale=.8]
		\draw[very thick] circle [radius=1];
		\foreach \t/\l in {30/P1,70/P2,110/P3,150/P4,-150/P5,-110/P6,-70/P7,-30/P8}
		\fill (\t:1) coordinate(\l) circle (1.8pt);
		\draw[thick] (P1) -- (P4);
		\draw[thick] (P2) -- (P7);
		\draw[thick] (P3) -- (P6);
		\draw[thick] (P5) -- (P8);
	\end{tikzpicture}
	& \qquad c^4 -4c^2 \epsilon y + 4cy - 4\epsilon y + 4y^2
	\\[5pt]
	\begin{tikzpicture}[baseline,scale=.8]
		\draw[very thick] circle [radius=1];
		\foreach \t/\l in {30/P1,70/P2,110/P3,150/P4,-150/P5,-110/P6,-70/P7,-30/P8}
		\fill (\t:1) coordinate(\l) circle (1.8pt);
		\draw[thick] (P2) -- (P7);
		\draw[thick] (P3) -- (P6);
		\draw[thick] (P1) -- (P5);
		\draw[thick] (P4) -- (P8);
	\end{tikzpicture}
	& \qquad c^4-5c^2 \epsilon y + 6cy - 5 \epsilon y + 4y^2
	\\[5pt]
	\begin{tikzpicture}[baseline,scale=.8]
		\draw[very thick] circle [radius=1];
		\foreach \t/\l in {0/P1,45/P2,90/P3,135/P4,180/P5,225/P6,270/P7,315/P8}
		\fill (\t:1) coordinate(\l) circle (1.8pt);
		\draw[thick] (P1) -- (P5);
		\draw[thick] (P2) -- (P6);
		\draw[thick] (P3) -- (P7);
		\draw[thick] (P4) -- (P8);
	\end{tikzpicture}
		& \qquad c^4 -6c^2 \epsilon y + 8cy - 7\epsilon y + 5y^2
\end{align}
\end{enumerate}

\section{Proofs} \label{SEC:Proof}
\setcounter{equation}{0}

We prepare some notations. We denote the basis and a structure constant of $ A1_{\epsilon} $ by $ Q_{\mu} $ and $ \f{\mu\nu}{\rho}, $ respectively, where $ \mu = 0, 1, 2, 3 $ (greek letters run from $0$ to $3$ and we reserve latin letters for $ 1, 2, 3$) and $ Q_0 := H. $ 
The $\Z2$ grading of $ Q_{\mu} $ is denoted by $ \hat{\mu} \in \Z2$ and $ \hat{\mu} \cdot \hat{\nu} $ is the inner product which defines the commuting factor \eqref{comfac} of $\Z2$-graded Lie algebra. 
Recalling that the invariant bilinear form of $A1_{\epsilon}$ is diagonal
\begin{equation}
	C = B^{-1} = \mathrm{diag}(1,\epsilon, 1, 1) \equiv \mathrm{diag}(b_0,b_1,b_2,b_3)
\end{equation} 
and noting that $ \f{\mu\nu}{\rho} = 0 $ if one of $\mu, \nu, \rho $ is zero, 
we define
\begin{equation}
	\tf{\mu\nu}{\rho} := \sum_{\alpha} \f{\mu\nu}{\alpha} C^{\alpha \rho} 
	= \f{\mu\nu}{\rho} b_{\rho} = 
	\begin{cases}
		1, & (\mu,\nu,\rho) = \text{cyclic perm. of } (1,2,3)
		\\[5pt]
		0, & \text{otherwise}
	\end{cases}
	\label{ftilddef}
\end{equation}
Namely, $ \tf{ij}{k} $ is cyclic and symmetric with respect to $(i, j, k) $ and $ (i, j), $ respectively. 
It is not difficult to see that 
\begin{align}
	\sum_{a} \tf{ia}{j} \tf{ak}{\ell} = \delta_{ik} \delta_{j\ell} - \ph{k}{\ell} \delta_{i\ell} \delta_{jk} 
	=\delta_{i\ell} \delta_{jk} - \ph{i}{j} \delta_{ik} \delta_{j\ell}. 
	\label{ftildrel}
\end{align}
Let us recall that for a $\Z2$-graded Lie algebra, the mapping $S$ \eqref{fScomp} is given by 
\begin{equation}
	S_{ij}^{k\ell} = \ph{i}{j} \delta_i^{\ell} \delta_j^k
	\quad \longleftrightarrow \quad
	 \begin{tikzpicture}[baseline=13pt]
	 	\draw[thick] (0,0) node[below] {$i$} -- (1,1) node[above] {$i$};
	 	\draw[thick] (0,1) node[above] {$j$} -- (1,0) node[below] {$j$};
	 	\node at (1.8,0.5) {$\ph{i}{j}$};
	 \end{tikzpicture} 
	 \label{SforZ22}
\end{equation}
i.e., the crossing has the phase factor $\ph{i}{j}$. 

Throughout this section, we write $w(D)$ as $D$, i.e., we omit $w$ and a diagram implies its weight. 

\subsection{Proof of Proposition \ref{PROP:Y}}

Draw the diagram as follows (left), then cut and open it at the position of the arrow which gives the diagram on the right:
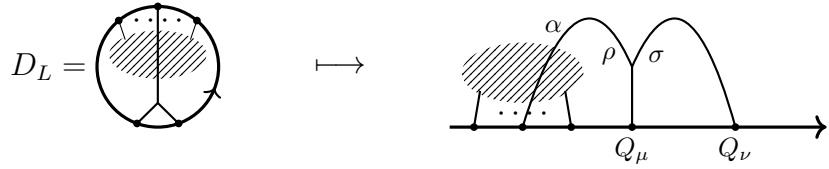
\begin{center}
\begin{tikzpicture}[baseline=0,scale=0.8]
	\draw[very thick] (0,0) circle (1);
	\draw (50:1) coordinate(A) -- (40:0.7);
	\draw (130:1) coordinate(B)-- (140:0.7);
	\node at (90:0.76) {$\cdots\cdot$};
	\fill[pattern=north east lines] (0,0.2) ellipse (.8 and .4);
	\draw[thick] (0,1) coordinate(E)--(0,-0.6) coordinate(v);
	\draw[thick] (v)--(-70:1) coordinate(C);
	\draw[thick] (v)--(-110:1) coordinate(D);
	\foreach \l in {A,B,C,D,E}
		\fill (\l) circle (1.8pt);
	\draw[thick,->] (330:1) arc (330:340:1);
	\node at (3,0) {$\longmapsto$};
	\node at (-1.8,0) {$D_L = $};
	\begin{scope}[xshift=5cm]
		\draw[very thick, ->] (-0.2,-1) -- (6,-1);
		\fill[pattern=north east lines] (1,-0.1) ellipse (1 and .5);
		\draw[thick] (0.2,-1) coordinate(A) -- (0.3,-0.4);
		\draw[thick] (1.8,-1) coordinate(B) -- (1.7,-0.4);
		\node[yshift=5pt] at ($(A)!0.5!(B)$) {$\cdots\cdot$};
		\draw[thick] (1,-1) coordinate(E)  parabola[bend at end] (2.1,0.8) node[left,xshift=-6pt,yshift=-3pt] {\footnotesize $\alpha$} parabola (2.8,0)  coordinate (v) parabola[bend at end] (3.5,0.8) parabola (4.5,-1) coordinate(C);
		\draw[thick] (v)--($(v)+(0,-1)$) coordinate(D) ;
		\node[left,xshift=-2pt,yshift=4pt] at (v) {\footnotesize $\rho$};
		\node[right,xshift=2pt,yshift=4pt] at (v) {\footnotesize $\sigma$};
		\foreach \l in {A,B,C,D,E}
			\fill (\l) circle (1.8pt);
		\node[below] at (C) {\footnotesize $Q_{\nu}$};
		\node[below] at (D)  {\footnotesize $Q_{\mu}$};
	\end{scope}
\end{tikzpicture}
\end{center}
where the shaded region indicates all other edges. 
Following the construction given in \S \ref{SEC:UWSL}, we get
\begin{align}
	D_L &= 
	\sum A_{\alpha} C^{\alpha\rho} C^{\sigma \nu} \f{\rho\sigma}{\mu} Q_{\mu} Q_{\nu}
	=
	\sum A_i b_i b_j \f{ij}{k} Q_{k} Q_j
	\notag \\
	& = A_1 Q_1 + \epsilon (A_2 Q_2 + A_3 Q_3)
\end{align}
where $ A_a \in \mathcal{U}(A1_{\epsilon}) $ is the contribution from the shaded part and the explicit expression of $f$ was used. 

Similarly, one gets for the diagram
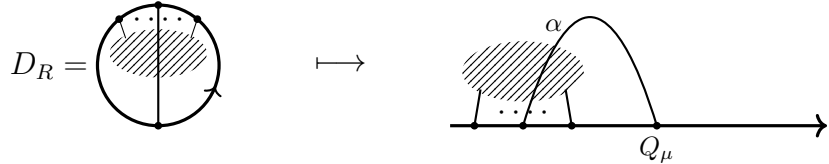
\begin{center}
	\begin{tikzpicture}[baseline=0,scale=0.8]
		\draw[very thick] (0,0) circle (1);
		\draw (50:1) coordinate(A) -- (40:0.7);
		\draw (130:1) coordinate(B)-- (140:0.7);
		\node at (90:0.76) {$\cdots\cdot$};
		\fill[pattern=north east lines] (0,0.2) ellipse (.8 and .4);
		\draw[thick] (0,1) coordinate(C)--(0,-1) coordinate(D);
		\foreach \l in {A,B,C,D}
			\fill (\l) circle (1.8pt);
		\draw[thick,->] (330:1) arc (330:340:1);
		\node at (3,0) {$\longmapsto$};
		\node at (-1.8,0) {$D_R = $};
		\begin{scope}[xshift=5cm]
			\draw[very thick, ->] (-0.2,-1) -- (6,-1);
			\fill[pattern=north east lines] (1,-0.1) ellipse (1 and .5);
			\draw[thick] (0.2,-1) coordinate(A) -- (0.3,-0.4);
			\draw[thick] (1.8,-1) coordinate(B) -- (1.7,-0.4);
			\node[yshift=5pt] at ($(A)!0.5!(B)$) {$\cdots\cdot$};
			\draw[thick] (1,-1) coordinate(C)  parabola[bend at end] (2.1,0.8) node[left,xshift=-6pt,yshift=-3pt] {\footnotesize $\alpha$} parabola (3.2,-1) coordinate(D);
			\node[below] at (D) {\footnotesize $Q_{\mu}$}; 
			\foreach \l in {A,B,C,D}
				\fill (\l) circle (1.8pt);
		\end{scope}
	\end{tikzpicture}
\end{center}
\begin{equation}
	D_R = \sum A_{\alpha} C^{\alpha \mu} Q_{\mu} = A_0 Q_0 + \epsilon A_1 Q_1 + A_2Q_2 + A_3 Q_3 
	= A_0 Q_0 + \epsilon D_L.
\end{equation}
Since $Q_0$ is the center of $ A1_{\epsilon}$, we have
\begin{center}
	\begin{tikzpicture}[baseline,scale=0.8]
		\draw[very thick, ->] (-0.2,-1) -- (4.5,-1);
		\fill[pattern=north east lines] (1,-0.1) ellipse (1 and .5);
		\draw[thick] (0.2,-1) coordinate(A) -- (0.3,-0.4);
		\draw[thick] (1.8,-1) coordinate(B) -- (1.7,-0.4);
		\node[yshift=5pt] at ($(A)!0.5!(B)$) {$\cdots\cdot$};
		\draw[thick] (1,-1) coordinate(C) parabola[bend at end] (2.1,0.8) node[left,xshift=-6pt,yshift=-3pt] {\footnotesize $0$} parabola (3.2,-1) coordinate(D);
		\node[below] at (C) {\footnotesize $Q_0$};
		\node[below] at (D) {\footnotesize $Q_{0}$};
		\foreach \l in {A,B,C,D}
			\fill (\l) circle (1.8pt);
		\node at (5.2,-0.3) {$=$};
		\node at (-1.5,-0.3) {$A_0 Q_0 = $};
		\begin{scope}[xshift=6.2cm]
			\draw[very thick, ->] (-0.2,-1) -- (5,-1);
			\fill[pattern=north east lines] (1,-0.1) ellipse (1 and .5);
			\draw[thick] (0.2,-1) coordinate(A) -- (0.3,-0.4);
			\draw[thick] (1.8,-1) coordinate(B) -- (1.7,-0.4);
			\node[yshift=5pt] at ($(A)!0.5!(B)$) {$\cdots\cdot$};
			\draw[thick] (2.4,-1) coordinate(C) parabola[bend at end] (3.3,0.8)  parabola (4.2,-1) coordinate(D); 
			\node[below] at (C) {\footnotesize $Q_0$};
			\node[below] at (D) {\footnotesize $Q_0$};
			\foreach \l in {A,B,C,D}
				\fill (\l) circle (1.8pt);
			\node at (5.5,-0.3) {$=$};
		\end{scope}
		\begin{scope}[xshift=13.5cm,yshift=-7]
			\draw[very thick] (0,0) circle (1);
			\draw (50:1) coordinate(A) -- (40:0.7);
			\draw (130:1) coordinate(B)-- (140:0.7);
			\node at (90:0.76) {$\cdots\cdot$};
			\fill[pattern=north east lines] (0,0.2) ellipse (.8 and .4);
			\foreach \l in {A,B}
				\fill (\l) circle (1.8pt);
			\draw[thick,->] (330:1) arc (330:340:1);
			\node[right] at (1,0) {$\cdot Q_0^2$};
		\end{scope}
	\end{tikzpicture}
\end{center}
Recall $Q_0^2 = c-y$, this proves the proposition.

\subsection{Proof of Proposition \ref{PROP:geta} and Proposition \ref{PROP:cross}}

The weight of the LHS of \eqref{Geta} is evaluated as follows:
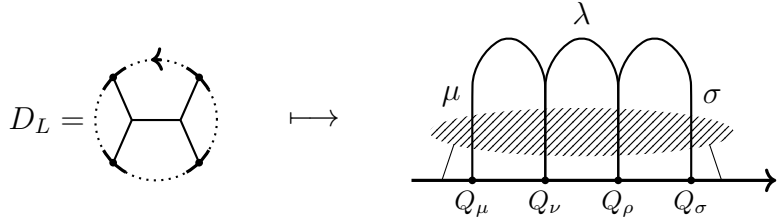
\begin{center}
	\begin{tikzpicture}[baseline=0,scale=.8]
		\draw[dotted,thick] (0,0) circle [radius=1];
		\foreach \t/\l in {45/A,135/B,225/C,315/D}{
			\fill (\t:1) circle (1.8pt) coordinate (\l);
			\draw[very thick] (\t-15:1) arc(\t-15:\t+15:1);
		}
		\coordinate (v1) at (-0.4,0); \coordinate (v2) at (0.4,0);
		\draw[thick] (v1)--(v2);
		\draw[thick] (A)--(v2); \draw[thick] (D)--(v2);
		\draw[thick] (B)--(v1); \draw[thick] (C)--(v1);
		\draw[very thick,->] (85:1) arc (85:95:1);
		\node at (2.6,0) {$\longmapsto$};
		\node at (-1.8,0) {$D_L = $};
		\begin{scope}[xshift=5.2cm]
			\draw[very thick, ->] (-1,-1) -- (5,-1);
			\foreach \t/\l in {0/A,1.2/B,2.4/C,3.6/D}{
				\fill (\t,-1) circle (1.8pt) coordinate(\l);
			}
			\draw[thick,rounded corners] (A) -- ($(A)+(0,1.4)$) node[left] {$\mu$} to[bend left] ($(A)!0.5!(B)+(0,2.4)$) to[bend left] ($(B)+(0,1.4)$) -- (B);
			\draw[thick,rounded corners] (B) -- ($(B)+(0,1.4)$) to[bend left] ($(B)!0.5!(C)+(0,2.4)$) node[above] {$\lambda$} to[bend left] ($(C)+(0,1.4)$) -- (C);
			\draw[thick,rounded corners] (C) -- ($(C)+(0,1.4)$) to[bend left] ($(C)!0.5!(D)+(0,2.4)$)  to[bend left] ($(D)+(0,1.4)$) node[right] {$\sigma$} -- (D);
			\fill[pattern=north east lines] ($(A)!0.5!(D)+(0,.8)$) ellipse (2.5 and .4);
			\draw (-0.5,-1)--(-0.3,-0.4);  \draw (4.1,-1)--(3.9,-0.4);
			\foreach \m/\l in {\mu/A,\nu/B,\rho/C,\sigma/D}
			\node[below] at (\l) {\footnotesize $Q_{\m}$};
		\end{scope}
	\end{tikzpicture}
\end{center}

\begin{align}
	D_L &= \sum b_{\mu} b_{\lambda} b_{\sigma} \f{\mu\lambda}{\nu} \f{\lambda\sigma}{\rho} 
		A Q_{\mu} T Q_{\nu} U Q_{\rho} V Q_{\sigma} W
		\notag \\
		&= \sum b_i b_a b_j b_k b_{\ell} \tf{ia}{j} \tf{a\ell}{k} 
		A Q_i T Q_j U Q_k V Q_{\ell} W
\end{align}
where $A, T, U, V, W \in \mathcal{U}(A1_{\epsilon}) $ are the contributions from the shaded region. 

\medskip\noindent
\textbf{Remarks:}
\begin{enumerate}
	\item The labels of the edges before and after through the shaded region are the same, because of \eqref{SforZ22}.
	\item Due to the phase factor in \eqref{SforZ22}, $A, T, U, V, W $ depend on the labels of the edges through the shaded region, but we do not write it explicitly. 
\end{enumerate}
Noting \eqref{ftilddef}, one may readily see that $ b_i b_a b_j \tf{ia}{j} = \epsilon \tf{ia}{j}. $ Using \eqref{ftildrel}, we may finally obtain the following expression for $D_L$:
\begin{align}
	D_L = \sum \epsilon b_i b_j \big( A Q_i T Q_j U Q_j V Q_i W 
	  - \ph{i}{j} A Q_i T Q_j U Q_i V Q_j W \big).
\end{align}

We now turn to the evaluation of the RHS of \eqref{Geta}. 
First, let us consider the diagrams with two chords:
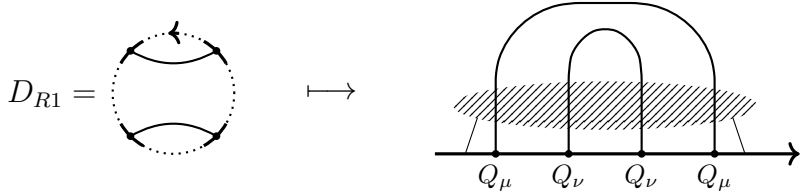
\begin{center}
	\begin{tikzpicture}[baseline=0,scale=.8]
		\draw[dotted,thick] (0,0) circle [radius=1];
		\foreach \t/\l in {45/A,135/B,225/C,315/D}{
			\fill (\t:1) circle (1.8pt) coordinate (\l);
			\draw[very thick] (\t-15:1) arc(\t-15:\t+15:1);
		}
		\draw[thick] (A)to[bend left] (B); \draw[thick] (C) to[bend left] (D);
		\draw[very thick,->] (85:1) arc (85:95:1);
		\node at (2.6,0) {$\longmapsto$};
		\node at (-2,0) {$D_{R1} = $};
		\begin{scope}[xshift=5.3cm]
			\draw[very thick, ->] (-1,-1) -- (5,-1);
			\foreach \t/\l in {0/A,1.2/B,2.4/C,3.6/D}{
				\coordinate (\l) at (\t,.4);
				\fill (\t,-1) circle (1.8pt) coordinate(\l-1);
			}
			\draw[thick,rounded corners] (A-1) -- (0,0.4) to[bend left] (1.2,1.5) to (2.4,1.5) to[bend left] (3.6,0.4) -- (D-1);
			\draw[thick,rounded corners] (B-1) -- (1.2,0.4) to[bend left] ($(B)!0.5!(C)+(0,0.7)$) to[bend left] (2.4,0.4) -- (C-1);
			\fill[pattern=north east lines] ($(A)!0.5!(D)+(0,-0.6)$) ellipse (2.5 and .4);
			\draw (-0.5,-1)--(-0.3,-0.4);  \draw (4.1,-1)--(3.9,-0.4);
			\foreach \m/\l in {\mu/A,\nu/B,\nu/C,\mu/D}
			\node[below] at (\l-1) {\footnotesize $Q_{\m}$};
		\end{scope}
	\end{tikzpicture}
\end{center}
\begin{align}
	D_{R1} &= \sum b_{\mu} b_{\nu} A Q_{\mu} T Q_{\nu} U Q_{\nu} V Q_{\mu} W
	\notag \\
	&= Q_0^4 ATUVW + \sum b_i Q_0^2 (AT Q_i U Q_i VW + A Q_i TUV Q_i W) 
	\notag \\
	&+ \sum b_i b_j A Q_i T Q_j U Q_j V Q_i W.
\end{align}
Each $A, T, U, V $  and $W$ in the second and third lines is not the same, they differ by the phase factor coming from \eqref{SforZ22}. We use the same symbols for all of them, but this does not harm the proof. 
Similarly, we have
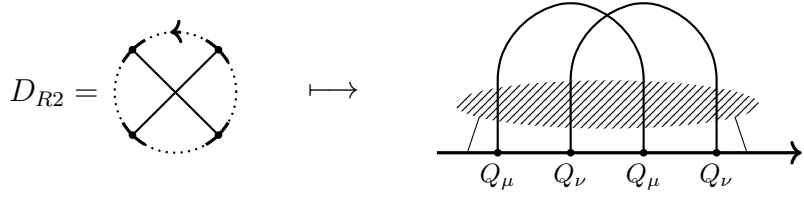
\begin{center}	
	\begin{tikzpicture}[baseline=0,scale=.8]
		\draw[dotted,thick] (0,0) circle [radius=1];
		\foreach \t/\l in {45/A,135/B,225/C,315/D}{
			\fill (\t:1) circle (1.8pt) coordinate (\l);
			\draw[very thick] (\t-15:1) arc(\t-15:\t+15:1);
		}
		\draw[thick] (A)--(C); \draw[thick] (B)--(D);
		\draw[very thick,->] (85:1) arc (85:95:1);
		\node at (2.6,0) {$\longmapsto$};
		\node at (-2,0) {$D_{R2} = $};
		\begin{scope}[xshift=5.3cm]
			\draw[very thick, ->] (-1,-1) -- (5,-1);
			\foreach \t/\l in {0/A,1.2/B,2.4/C,3.6/D}{
				\coordinate (\l) at (\t,.4);
				\fill (\t,-1) circle (1.8pt) coordinate(\l-1);
			}
			\draw[thick,rounded corners] (A-1) -- (0,0.4) to[bend left] (1.2,1.5) to[bend left] (2.4,0.4) --  (C-1);
			\draw[thick,rounded corners] (B-1) -- (1.2,0.4) to[bend left] (2.4,1.5) to[bend left] (3.6,0.4) -- (D-1);
			\fill[pattern=north east lines] ($(A)!0.5!(D)+(0,-0.6)$) ellipse (2.5 and .4);
			\draw (-0.5,-1)--(-0.3,-0.4);  \draw (4.1,-1)--(3.9,-0.4);
			\foreach \m/\l in {\mu/A,\nu/B,\mu/C,\nu/D}
			\node[below] at (\l-1) {\footnotesize $Q_{\m}$};
		\end{scope}
	\end{tikzpicture}
\end{center}
\begin{align}
	D_{R2} &= \sum \ph{\mu}{\nu} b_{\mu} b_{\nu} A Q_{\mu} T Q_{\nu} U Q_{\mu} V Q_{\nu} W
	\notag \\
	&= Q_0^4 ATUVW + \sum b_i Q_0^2 (AT Q_i UV Q_i W + A Q_i TU Q_i VW) 
	\notag \\
	&+ \sum \ph{i}{j} b_i b_j A Q_i T Q_j U Q_i V Q_j W.
\end{align}
It follows that
\begin{align}
	D_{R1} - D_{R2} = \epsilon D_L 
	&+ \sum b_i Q_0^2 (AT Q_i U Q_i VW + A Q_i TUV Q_i W) 
	\notag \\[3pt]
	&-\sum b_i Q_0^2 (AT Q_i UV Q_i W + A Q_i TU Q_i VW). 
	\label{DR1mDR2}
\end{align}

Next, we consider the diagrams with one chord. 
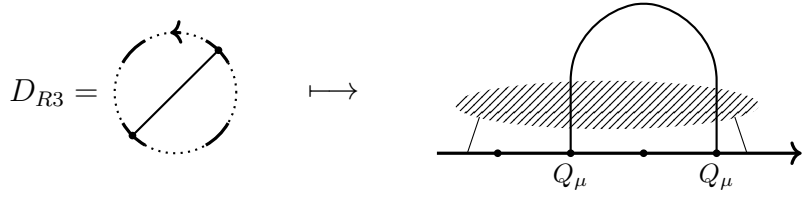
\begin{center}
	\begin{tikzpicture}[baseline=0,scale=.8]
		\draw[dotted,thick] (0,0) circle [radius=1];
		\foreach \t/\l in {45/A,135/B,225/C,315/D}{
			\coordinate (\l) at (\t:1);
			\draw[very thick] (\t-15:1) arc(\t-15:\t+15:1);
		}
		\draw[thick] (C)--(A); 
		\fill (A) circle (1.8pt); \fill (C) circle (1.8pt);
		\draw[very thick,->] (85:1) arc (85:95:1);
		\node at (2.6,0) {$\longmapsto$};
		\node at (-2,0) {$D_{R3} = $};
		\begin{scope}[xshift=5.3cm]
			\draw[very thick, ->] (-1,-1) -- (5,-1);
			\foreach \t/\l in {0/A,1.2/B,2.4/C,3.6/D}{
				\coordinate (\l) at (\t,.4);
				\fill (\t,-1) circle (1.8pt) coordinate(\l-1);
			}
	\draw[thick,rounded corners] (B-1) -- (1.2,0.4) to[bend left] (2.4,1.5) to[bend left] (3.6,0.4) -- (D-1);
			\fill[pattern=north east lines] ($(A)!0.5!(D)+(0,-0.6)$) ellipse (2.5 and .4);
			\draw (-0.5,-1)--(-0.3,-0.4);  \draw (4.1,-1)--(3.9,-0.4);
			\node[below] at (B-1) {\footnotesize $Q_{\mu}$};
			\node[below] at (D-1) {\footnotesize $Q_{\mu}$};
		\end{scope}
	\end{tikzpicture}
\end{center}
\begin{align}
	D_{R3} = \sum b_{\mu} AT Q_{\mu} UV Q_{\mu} W
	= Q_0^2 ATUVW + \sum b_i AT Q_i UV Q_i W. 
\end{align}
In a similar way, we get the following:
\begin{align*}
	\begin{tikzpicture}[baseline=0,scale=.8]
		\draw[dotted,thick] (0,0) circle [radius=1];
		\foreach \t/\l in {45/A,135/B,225/C,315/D}{
			\coordinate (\l) at (\t:1);
			\draw[very thick] (\t-15:1) arc(\t-15:\t+15:1);
		}
		\draw[thick] (B)--(D); 
		\fill (B) circle (1.8pt); \fill (D) circle (1.8pt);
		\node at (-2,0) {$D_{R4} = $};
	\end{tikzpicture}
	\ &= Q_0^2 ATUVW + \sum b_i A Q_i TU Q_i VW,
	\\
	\begin{tikzpicture}[baseline=0,scale=.8]
		\draw[dotted,thick] (0,0) circle [radius=1];
		\foreach \t/\l in {45/A,135/B,225/C,315/D}{
			\coordinate (\l) at (\t:1);
			\draw[very thick] (\t-15:1) arc(\t-15:\t+15:1);
		}
		\draw[thick] (A) to[bend left] (B); 
		\fill (A) circle (1.8pt); \fill (B) circle (1.8pt);
		\node at (-2,0) {$D_{R5} = $};
	\end{tikzpicture}
	\ &= Q_0^2 ATUVW + \sum b_i A Q_i TUV Q_i W,
	\\
	\begin{tikzpicture}[baseline=0,scale=.8]
		\draw[dotted,thick] (0,0) circle [radius=1];
		\foreach \t/\l in {45/A,135/B,225/C,315/D}{
			\coordinate (\l) at (\t:1);
			\draw[very thick] (\t-15:1) arc(\t-15:\t+15:1);
		}
		\draw[thick] (C) to[bend left] (D); 
		\fill (D) circle (1.8pt); \fill (C) circle (1.8pt);
		\node at (-2,0) {$D_{R6} = $};
	\end{tikzpicture}
	\ &= Q_0^2 ATUVW + \sum b_i AT Q_i U Q_i VW.
\end{align*}
Therefore, \eqref{DR1mDR2} becomes
\begin{equation}
	D_{R1} - D_{R2} = \epsilon D_L + Q_0^2 (D_{R5} + D_{R6} - D_{R3} - D_{R4}) 
\end{equation}
which proves Proposition \ref{PROP:geta}. 

Proposition \ref{PROP:cross} can be proved similarly. 
Here we give the weight of the diagrams which do not appear in Proposition \ref{PROP:geta} and omit the details of the proof.  
The weight of the diagram on the LHS is calculated by drawing the diagram as follow:
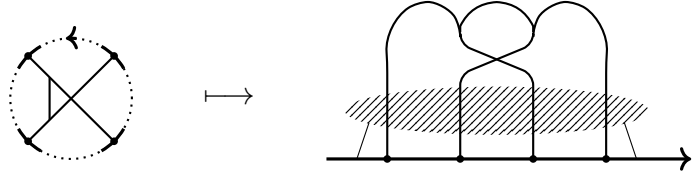
\begin{center}
	\begin{tikzpicture}[baseline=0,scale=.8]
		\draw[dotted,thick] (0,0) circle [radius=1];
		\foreach \t/\l in {45/A,135/B,225/C,315/D}{
			\fill (\t:1) circle (1.8pt) coordinate (\l);
			\draw[very thick] (\t-15:1) arc(\t-15:\t+15:1);
		}
		\draw[thick] (A)--(C);
		\draw[thick] (B)--(D); 
		\draw[thick] ($(A)!0.75!(C)$)--($(B)!0.25!(D)$);
		\draw[very thick,->] (85:1) arc (85:95:1);
		\node at (2.6,0) {$\longmapsto$};
		\begin{scope}[xshift=5.2cm]
			\draw[very thick, ->] (-1,-1) -- (5,-1);
			\foreach \t/\l in {0/A,1.2/B,2.4/C,3.6/D}{
				\fill (\t,-1) circle (1.8pt) coordinate(\l);
			}
			\draw[thick,rounded corners] (A) -- ($(A)+(0,1.6)$)  to[bend left] ($(A)!0.41!(B)+(0,2.6)$) to[bend left] ($(B)+(0,2.2)$) -- ($(B)+(0,1.9)$) -- ($(C)+(0,1.4)$) -- (C);
			\draw[thick,rounded corners] (D) -- ($(D)+(0,1.6)$)  to[bend right] ($(D)!0.41!(C)+(0,2.6)$) to[bend right] ($(C)+(0,2.2)$) -- ($(C)+(0,1.9)$) -- ($(B)+(0,1.4)$) -- (B);
			\draw[thick] ($(B)+(0,2.1)$) -- ($(B)+(0,2.2)$) to[bend left] ($(B)!0.5!(C)+(0,2.6)$) to[bend left] ($(C)+(0,2.2)$) -- ($(C)+(0,2.1)$);
			\fill[pattern=north east lines] ($(A)!0.5!(D)+(0,.8)$) ellipse (2.5 and .4);
			\draw (-0.5,-1)--(-0.3,-0.4);  \draw (4.1,-1)--(3.9,-0.4);
		\end{scope}
	\end{tikzpicture}
\end{center}
The weight is given by
\begin{equation}
	\sum \epsilon b_i b_j (A Q_i T Q_j U  Q_j V Q_i W - A Q_i T Q_i U Q_j V Q_j W).
\end{equation}
The weights of some diagram on the RHS are given by
\begin{align}
	\begin{tikzpicture}[baseline=-3pt,scale=.8]
		\draw[dotted,thick] (0,0) circle [radius=1];
		\foreach \t/\l in {45/A,135/B,225/C,315/D}{
			\fill (\t:1) circle (1.8pt) coordinate (\l);
			\draw[very thick] (\t-15:1) arc(\t-15:\t+15:1);
		}
		\draw[thick] (A)to[bend right] (D); \draw[thick] (B) to[bend left] (C);
	\end{tikzpicture}
	\ &= Q_0^4 ATUVW + \sum b_i  Q_0^2 (ATU Q_i V Q_i W + A Q_i T Q_i UVW)
	\notag \\
	&+ \sum b_i b_j A Q_i T Q_i U Q_j V Q_j W,
	\notag \\
	\begin{tikzpicture}[baseline=-3pt,scale=.8]
		\draw[dotted,thick] (0,0) circle [radius=1];
		\foreach \t/\l in {45/A,135/B,225/C,315/D}{
			\coordinate (\l) at (\t:1);
			\draw[very thick] (\t-15:1) arc(\t-15:\t+15:1);
		}
		\draw[thick] (B) to[bend left] (C); 
		\fill (B) circle (1.8pt); \fill (C) circle (1.8pt);
	\end{tikzpicture} 
	\ &= Q_0^2 ATUVW + \sum b_i A Q_i T Q_i UVW,
	\notag \\
	\begin{tikzpicture}[baseline=-3pt,scale=.8]
		\draw[dotted,thick] (0,0) circle [radius=1];
		\foreach \t/\l in {45/A,135/B,225/C,315/D}{
			\coordinate (\l) at (\t:1);
			\draw[very thick] (\t-15:1) arc(\t-15:\t+15:1);
		}
		\draw[thick] (D) to[bend left] (A); 
		\fill (D) circle (1.8pt); \fill (A) circle (1.8pt);
	\end{tikzpicture}	
	\ &= Q_0^2 ATUVW + \sum b_i ATU Q_i V Q_i W.
\end{align}

\subsection{Proof of Theorem \ref{THM:recrel}}

This theorem can be proved in a  similar way to $sl(2)$ \cite{ChVar} and $gl(1|1)$ \cite{FOKV}, namely by the induction on the number of endpoints of chords. 
The following lemma will be used in the proof. 

\begin{lemma} \label{LEM:ten}
	The following relations hold true:
\begin{align}
	\begin{tikzpicture}[scale=0.8,baseline=-3pt]
		\draw[dotted,thick] (0,0) circle [radius=1];
		\foreach \t/\l in {45/A,135/B,225/C,315/D}{
			\coordinate (\l) at (\t:1);
			\coordinate (\l-1) at (\t-10:1);
			\coordinate (\l+1) at (\t+10:1); 
			\draw[very thick] (\t-20:1) arc (\t-20:\t+20:1);
		}
		\draw[thick] (A) to[bend left] (B+1);
		\draw[thick] (B-1) to[bend left] (C+1);
		\draw[thick] (C-1) to[bend left] (D);
		\foreach \l in {A,B+1,B-1,C+1,C-1,D}
		\fill (\l) circle (1.8pt);
	\end{tikzpicture}
	&-
	\begin{tikzpicture}[scale=0.8,baseline=-3pt]
		\draw[dotted,thick] (0,0) circle [radius=1];
		\foreach \t/\l in {45/A,135/B,225/C,315/D}{
			\coordinate (\l) at (\t:1);
			\coordinate (\l-1) at (\t-10:1);
			\coordinate (\l+1) at (\t+10:1); 
			\draw[very thick] (\t-20:1) arc (\t-20:\t+20:1);
		}
		\draw[thick] (A) to[bend left] (B+1);
		\draw[thick] (B-1) to[bend left] (C-1);
		\draw[thick] (C+1) to[bend left] (D);
		\foreach \l in {A,B+1,B-1,C+1,C-1,D}
		\fill (\l) circle (1.8pt);
	\end{tikzpicture}
	-
	\begin{tikzpicture}[scale=0.8,baseline=-3pt]
		\draw[dotted,thick] (0,0) circle [radius=1];
		\foreach \t/\l in {45/A,135/B,225/C,315/D}{
			\coordinate (\l) at (\t:1);
			\coordinate (\l-1) at (\t-10:1);
			\coordinate (\l+1) at (\t+10:1); 
			\draw[very thick] (\t-20:1) arc (\t-20:\t+20:1);
		}
		\draw[thick] (A) to[bend left] (B-1);
		\draw[thick] (B+1) to[bend left] (C+1);
		\draw[thick] (C-1) to[bend left] (D);
		\foreach \l in {A,B+1,B-1,C+1,C-1,D}
		\fill (\l) circle (1.8pt);
	\end{tikzpicture}
	+
	\begin{tikzpicture}[scale=0.8,baseline=-3pt]
		\draw[dotted,thick] (0,0) circle [radius=1];
		\foreach \t/\l in {45/A,135/B,225/C,315/D}{
			\coordinate (\l) at (\t:1);
			\coordinate (\l-1) at (\t-10:1);
			\coordinate (\l+1) at (\t+10:1); 
			\draw[very thick] (\t-20:1) arc (\t-20:\t+20:1);
		}
		\draw[thick] (A) to[bend left] (B-1);
		\draw[thick] (B+1) to[bend left] (C-1);
		\draw[thick] (C+1) to[bend left] (D);
		\foreach \l in {A,B+1,B-1,C+1,C-1,D}
		\fill (\l) circle (1.8pt);
	\end{tikzpicture}
	= \epsilon \Lambda + \epsilon (c-y) \Gamma, 
	\label{ten1}
\end{align}
%
%
\begin{align}
	\begin{tikzpicture}[scale=0.8,baseline=-3pt]
		\draw[dotted,thick] (0,0) circle [radius=1];
		\foreach \t/\l in {35/A,135/B,225/C,325/D}{
			\coordinate (\l) at (\t:1);
			\coordinate (\l-1) at (\t-10:1);
			\coordinate (\l+1) at (\t+10:1); 
			\draw[very thick] (\t-20:1) arc (\t-20:\t+20:1);
		}
		\draw[thick] (A) -- (C-1);
		\draw[thick] (B-1) to[bend left] (C+1);
		\draw[thick] (B+1) -- (D);
		\foreach \l in {A,B+1,B-1,C+1,C-1,D}
		\fill (\l) circle (1.8pt);
	\end{tikzpicture}
	&-
	\begin{tikzpicture}[scale=0.8,baseline=-3pt]
		\draw[dotted,thick] (0,0) circle [radius=1];
		\foreach \t/\l in {35/A,135/B,225/C,325/D}{
			\coordinate (\l) at (\t:1);
			\coordinate (\l-1) at (\t-10:1);
			\coordinate (\l+1) at (\t+10:1); 
			\draw[very thick] (\t-20:1) arc (\t-20:\t+20:1);
		}
		\draw[thick] (A) -- (C-1);
		\draw[thick] (B+1) to[bend left] (C+1);
		\draw[thick] (B-1) -- (D);
		\foreach \l in {A,B+1,B-1,C+1,C-1,D}
		\fill (\l) circle (1.8pt);
	\end{tikzpicture}
	-
	\begin{tikzpicture}[scale=0.8,baseline=-3pt]
		\draw[dotted,thick] (0,0) circle [radius=1];
		\foreach \t/\l in {35/A,135/B,225/C,325/D}{
			\coordinate (\l) at (\t:1);
			\coordinate (\l-1) at (\t-10:1);
			\coordinate (\l+1) at (\t+10:1); 
			\draw[very thick] (\t-20:1) arc (\t-20:\t+20:1);
		}
		\draw[thick] (A) -- (C+1);
		\draw[thick] (B-1) to[bend left] (C-1);
		\draw[thick] (B+1) -- (D);
		\foreach \l in {A,B+1,B-1,C+1,C-1,D}
		\fill (\l) circle (1.8pt);
	\end{tikzpicture}
	+
	\begin{tikzpicture}[scale=0.8,baseline=-3pt]
		\draw[dotted,thick] (0,0) circle [radius=1];
		\foreach \t/\l in {35/A,135/B,225/C,325/D}{
			\coordinate (\l) at (\t:1);
			\coordinate (\l-1) at (\t-10:1);
			\coordinate (\l+1) at (\t+10:1); 
			\draw[very thick] (\t-20:1) arc (\t-20:\t+20:1);
		}
		\draw[thick] (A) -- (C+1);
		\draw[thick] (B+1) to[bend left] (C-1);
		\draw[thick] (B-1) -- (D);
		\foreach \l in {A,B+1,B-1,C+1,C-1,D}
		\fill (\l) circle (1.8pt);
	\end{tikzpicture}
	= \epsilon \left(
	\begin{tikzpicture}[scale=0.8,baseline=-3pt]
		\draw[dotted,thick] (0,0) circle [radius=1];
		\foreach \t/\l in {45/A,135/B,225/C,315/D}{
			\coordinate (\l) at (\t:1);
			\fill (\l) circle (1.8pt);
			\draw[very thick] (\t-20:1) arc (\t-20:\t+20:1);
		}
		\draw[thick] (B) to[bend left] (C);
		\draw[thick] (A) to[bend right] (D);
	\end{tikzpicture}
	-
	\begin{tikzpicture}[scale=0.8,baseline=-3pt]
		\draw[dotted,thick] (0,0) circle [radius=1];
		\foreach \t/\l in {45/A,135/B,225/C,315/D}{
			\coordinate (\l) at (\t:1);
			\fill (\l) circle (1.8pt);
			\draw[very thick] (\t-20:1) arc (\t-20:\t+20:1);
		}
		\draw[thick] (A) to[bend left] (B);
		\draw[thick] (C) to[bend left] (D);
	\end{tikzpicture}
	\right)
	\notag \\
	& + \epsilon (c-y)
	\left(
	\begin{tikzpicture}[scale=0.8,baseline=-3pt]
		\draw[dotted,thick] (0,0) circle [radius=1];
		\foreach \t/\l in {45/A,135/B,225/C,315/D}{
			\coordinate (\l) at (\t:1);
			\draw[very thick] (\t-20:1) arc (\t-20:\t+20:1);
		}
		\draw[thick] (A) to[bend left] (B);
		\fill (A) circle (1.8pt) (B) circle (1.8pt);
	\end{tikzpicture}
	+
	\begin{tikzpicture}[scale=0.8,baseline=-3pt]
		\draw[dotted,thick] (0,0) circle [radius=1];
		\foreach \t/\l in {45/A,135/B,225/C,315/D}{
			\coordinate (\l) at (\t:1);
			\draw[very thick] (\t-20:1) arc (\t-20:\t+20:1);
		}
		\draw[thick] (C) to[bend left] (D);
		\fill (C) circle (1.8pt) (D) circle (1.8pt);
	\end{tikzpicture}
	-
	\begin{tikzpicture}[scale=0.8,baseline=-3pt]
		\draw[dotted,thick] (0,0) circle [radius=1];
		\foreach \t/\l in {45/A,135/B,225/C,315/D}{
			\coordinate (\l) at (\t:1);
			\draw[very thick] (\t-20:1) arc (\t-20:\t+20:1);
		}
		\draw[thick] (B) to[bend left] (C);
		\fill (B) circle (1.8pt) (C) circle (1.8pt);
	\end{tikzpicture}
	-
	\begin{tikzpicture}[scale=0.8,baseline=-3pt]
		\draw[dotted,thick] (0,0) circle [radius=1];
		\foreach \t/\l in {45/A,135/B,225/C,315/D}{
			\coordinate (\l) at (\t:1);
			\draw[very thick] (\t-20:1) arc (\t-20:\t+20:1);
		}
		\draw[thick] (A) to[bend right] (D);
		\fill (A) circle (1.8pt) (D) circle (1.8pt);
	\end{tikzpicture}
	\right),
	\label{ten2}
\end{align}
%
%
\begin{align}
	\begin{tikzpicture}[scale=0.8,baseline=-3pt]
		\draw[dotted,thick] (0,0) circle [radius=1];
		\foreach \t/\l in {35/A,135/B,225/C,325/D}{
			\coordinate (\l) at (\t:1);
			\coordinate (\l-1) at (\t-10:1);
			\coordinate (\l+1) at (\t+10:1); 
			\draw[very thick] (\t-20:1) arc (\t-20:\t+20:1);
		}
		\draw[thick] (A+1) -- (C+1);
		\draw[thick] (A-1) to[bend left] (B);
		\draw[thick] (C-1) to[bend left] (D);
		\foreach \l in {A-1,A+1,B,C-1,C+1,D}
			\fill (\l) circle (1.8pt);
	\end{tikzpicture}
	&-
	\begin{tikzpicture}[scale=0.8,baseline=-3pt]
		\draw[dotted,thick] (0,0) circle [radius=1];
		\foreach \t/\l in {35/A,135/B,225/C,325/D}{
			\coordinate (\l) at (\t:1);
			\coordinate (\l-1) at (\t-10:1);
			\coordinate (\l+1) at (\t+10:1); 
			\draw[very thick] (\t-20:1) arc (\t-20:\t+20:1);
		}
		\draw[thick] (A-1) -- (C+1);
		\draw[thick] (A+1) to[bend left] (B);
		\draw[thick] (C-1) to[bend left] (D);
		\foreach \l in {A-1,A+1,B,C-1,C+1,D}	
			\fill (\l) circle (1.8pt);
	\end{tikzpicture}
	-
	\begin{tikzpicture}[scale=0.8,baseline=-3pt]
		\draw[dotted,thick] (0,0) circle [radius=1];
		\foreach \t/\l in {35/A,135/B,225/C,325/D}{
			\coordinate (\l) at (\t:1);
			\coordinate (\l-1) at (\t-10:1);
			\coordinate (\l+1) at (\t+10:1); 
			\draw[very thick] (\t-20:1) arc (\t-20:\t+20:1);
		}
		\draw[thick] (A+1) -- (C-1);
		\draw[thick] (A-1) to[bend left] (B);
		\draw[thick] (C+1) to[bend left] (D);
		\foreach \l in {A-1,A+1,B,C-1,C+1,D}
			\fill (\l) circle (1.8pt);
	\end{tikzpicture}
	+
	\begin{tikzpicture}[scale=0.8,baseline=-3pt]
		\draw[dotted,thick] (0,0) circle [radius=1];
		\foreach \t/\l in {35/A,135/B,225/C,325/D}{
			\coordinate (\l) at (\t:1);
			\coordinate (\l-1) at (\t-10:1);
			\coordinate (\l+1) at (\t+10:1); 
			\draw[very thick] (\t-20:1) arc (\t-20:\t+20:1);
		}
		\draw[thick] (A-1) -- (C-1);
		\draw[thick] (A+1) to[bend left] (B);
		\draw[thick] (C+1) to[bend left] (D);
		\foreach \l in {A-1,A+1,B,C+1,C-1,D}
		\fill (\l) circle (1.8pt);
	\end{tikzpicture}
	= \epsilon \Lambda + \epsilon (c-y) \Gamma,
	\label{ten3}
\end{align}
%
%
\begin{align}
	\begin{tikzpicture}[scale=0.8,baseline=-3pt]
		\draw[dotted,thick] (0,0) circle [radius=1];
		\foreach \t/\l in {35/A,135/B,225/C,325/D}{
			\coordinate (\l) at (\t:1);
			\coordinate (\l-1) at (\t-10:1);
			\coordinate (\l+1) at (\t+10:1); 
			\draw[very thick] (\t-20:1) arc (\t-20:\t+20:1);
		}
		\draw[thick] (B-1) -- (D-1);
		\draw[thick] (A) to[bend left] (B+1);
		\draw[thick] (C) to[bend left] (D+1);
		\foreach \l in {A,B-1,B+1,C,D-1,D+1}
			\fill (\l) circle (1.8pt);
	\end{tikzpicture}
	&-
	\begin{tikzpicture}[scale=0.8,baseline=-3pt]
		\draw[dotted,thick] (0,0) circle [radius=1];
		\foreach \t/\l in {35/A,135/B,225/C,325/D}{
			\coordinate (\l) at (\t:1);
			\coordinate (\l-1) at (\t-10:1);
			\coordinate (\l+1) at (\t+10:1); 
			\draw[very thick] (\t-20:1) arc (\t-20:\t+20:1);
		}
		\draw[thick] (B+1) -- (D-1);
		\draw[thick] (A) to[bend left] (B-1);
		\draw[thick] (C) to[bend left] (D+1);
		\foreach \l in {A,B-1,B+1,C,D-1,D+1}	
			\fill (\l) circle (1.8pt);
	\end{tikzpicture}
	-
	\begin{tikzpicture}[scale=0.8,baseline=-3pt]
		\draw[dotted,thick] (0,0) circle [radius=1];
		\foreach \t/\l in {35/A,135/B,225/C,325/D}{
			\coordinate (\l) at (\t:1);
			\coordinate (\l-1) at (\t-10:1);
			\coordinate (\l+1) at (\t+10:1); 
			\draw[very thick] (\t-20:1) arc (\t-20:\t+20:1);
		}
		\draw[thick] (B-1) -- (D+1);
		\draw[thick] (A) to[bend left] (B+1);
		\draw[thick] (C) to[bend left] (D-1);
		\foreach \l in {A,B-1,B+1,C,D-1,D+1}
			\fill (\l) circle (1.8pt);
	\end{tikzpicture}
	+
	\begin{tikzpicture}[scale=0.8,baseline=-3pt]
		\draw[dotted,thick] (0,0) circle [radius=1];
		\foreach \t/\l in {35/A,135/B,225/C,325/D}{
			\coordinate (\l) at (\t:1);
			\coordinate (\l-1) at (\t-10:1);
			\coordinate (\l+1) at (\t+10:1); 
			\draw[very thick] (\t-20:1) arc (\t-20:\t+20:1);
		}
		\draw[thick] (B+1) -- (D+1);
		\draw[thick] (A) to[bend left] (B-1);
		\draw[thick] (C) to[bend left] (D-1);
		\foreach \l in {A,B-1,B+1,C,D-1,D+1}
			\fill (\l) circle (1.8pt);
	\end{tikzpicture}
	= \epsilon \Lambda + \epsilon (c-y) \Gamma 
	\label{ten4}
\end{align}
where
\begin{align}
	\Lambda &= 
	\begin{tikzpicture}[scale=0.8,baseline=-3pt]
		\draw[dotted,thick] (0,0) circle [radius=1];
		\foreach \t/\l in {45/A,135/B,225/C,315/D}{
			\coordinate (\l) at (\t:1);
			\fill (\l) circle (1.8pt);
			\draw[very thick] (\t-20:1) arc (\t-20:\t+20:1);
		}
		\draw[thick] (B) to[bend left] (C);
		\draw[thick] (A) to[bend right] (D);
	\end{tikzpicture}
	-
	\begin{tikzpicture}[scale=0.8,baseline=-3pt]
		\draw[dotted,thick] (0,0) circle [radius=1];
		\foreach \t/\l in {45/A,135/B,225/C,315/D}{
			\coordinate (\l) at (\t:1);
			\fill (\l) circle (1.8pt);
			\draw[very thick] (\t-20:1) arc (\t-20:\t+20:1);
		}
		\draw[thick] (A) -- (C);
		\draw[thick] (B) -- (D);
	\end{tikzpicture}
	\notag \\
	\Gamma &= 
	\begin{tikzpicture}[scale=0.8,baseline=-3pt]
		\draw[dotted,thick] (0,0) circle [radius=1];
		\foreach \t/\l in {45/A,135/B,225/C,315/D}{
			\coordinate (\l) at (\t:1);
			\draw[very thick] (\t-20:1) arc (\t-20:\t+20:1);
		}
		\draw[thick] (A) -- (C);
		\fill (A) circle (1.8pt) (C) circle (1.8pt);
	\end{tikzpicture}
	+
	\begin{tikzpicture}[scale=0.8,baseline=-3pt]
		\draw[dotted,thick] (0,0) circle [radius=1];
		\foreach \t/\l in {45/A,135/B,225/C,315/D}{
			\coordinate (\l) at (\t:1);
			\draw[very thick] (\t-20:1) arc (\t-20:\t+20:1);
		}
		\draw[thick] (B) -- (D);
		\fill (B) circle (1.8pt) (D) circle (1.8pt);
	\end{tikzpicture}
	-
	\begin{tikzpicture}[scale=0.8,baseline=-3pt]
		\draw[dotted,thick] (0,0) circle [radius=1];
		\foreach \t/\l in {45/A,135/B,225/C,315/D}{
			\coordinate (\l) at (\t:1);
			\draw[very thick] (\t-20:1) arc (\t-20:\t+20:1);
		}
		\draw[thick] (B) to[bend left] (C);
		\fill (B) circle (1.8pt) (C) circle (1.8pt);
	\end{tikzpicture}
	-
	\begin{tikzpicture}[scale=0.8,baseline=-3pt]
		\draw[dotted,thick] (0,0) circle [radius=1];
		\foreach \t/\l in {45/A,135/B,225/C,315/D}{
			\coordinate (\l) at (\t:1);
			\draw[very thick] (\t-20:1) arc (\t-20:\t+20:1);
		}
		\draw[thick] (A) to[bend right] (D);
		\fill (A) circle (1.8pt) (D) circle (1.8pt);
	\end{tikzpicture}
\end{align}
\end{lemma}
\begin{proof}
	The relations \eqref{ten1}, \eqref{ten3} and \eqref{ten4} are the direct consequence of Proposition \ref{PROP:geta}. The LHSs of them are obtained by applying the STU relation twice on the LHS of \eqref{Geta}. The relation \eqref{ten2} is obtained similarly from Proposition \ref{PROP:cross}. 
\end{proof}

As supposed in Theorem \ref{THM:recrel}, we take a chord ``$a$" in the chord diagram $D$ and draw $ D $ so that the chord $a$ is vertical.  
Then, each endpoint of any chord (except $a$) is either to the left or to the right  of the chord $a$.  The proof of Theorem \ref{THM:recrel} is an induction on $\ell,$ the number of endpoints on the left of $a$. We suppose that $\ell > 1, $ then there are seven configurations of chords whose endpoints are immediately to the left of $a$:

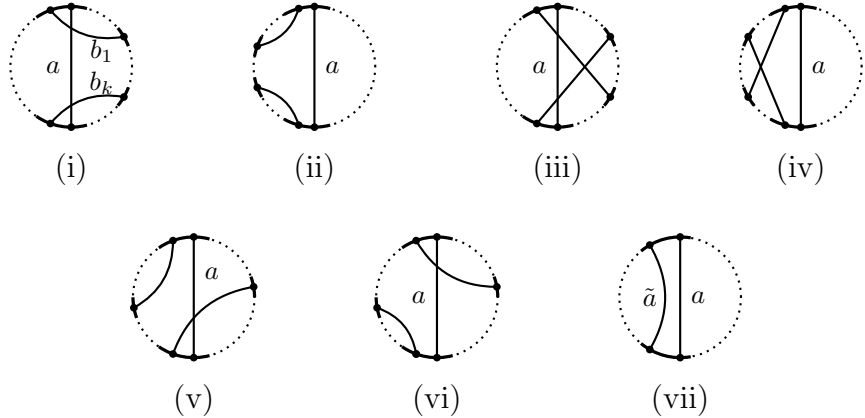
\begin{center}
	\begin{tikzpicture}[scale=.8]
		\draw[dotted,thick] (0,0) circle [radius=1];
		\foreach \t/\l in {30/A,-30/A1}{
			\coordinate (\l) at (\t:1);
			\draw[very thick] (\t-10:1) arc (\t-10:\t+10:1);
		}
		\foreach \t/\l in {90/B, -90/C}{
			\coordinate (\l) at (\t:1);
		}
		\draw[very thick] (75:1) arc (75:125:1);
		\draw[very thick] (-75:1) arc (-75:-125:1);
		\draw[thick] (B) -- (C);
		\node[left] at ($(B)!0.5!(C)$) {\footnotesize $a$};
		\draw[thick] (110:1) coordinate(D) to[bend right] (A);
			\node[below] at ($(D)!0.7!(A)$) {\footnotesize $b_1$};
		\draw[thick] (-110:1) coordinate(E) to[bend left] (A1);
			\node[above] at ($(E)!0.7!(A1)$) {\footnotesize $b_k$};
		\foreach \l in {A,A1,B,C,D,E}
			\fill (\l) circle (1.8pt);
		\node at (0,-1.7) {(i)};
		\begin{scope}[xshift=4cm]
			\draw[dotted,thick] (0,0) circle [radius=1];
			\foreach \t/\l in {160/A,-160/A1}{
				\coordinate (\l) at (\t:1);
				\draw[very thick] (\t-10:1) arc (\t-10:\t+10:1);
			}
			\foreach \t/\l in {90/B, -90/C}{
				\coordinate (\l) at (\t:1);
			}
			\draw[very thick] (75:1) arc (75:120:1);
			\draw[very thick] (-75:1) arc (-75:-120:1);
			\draw[thick] (B) -- (C);
			\node[right] at ($(B)!0.5!(C)$) {\footnotesize $a$};
			\draw[thick] (A) to[bend right] (105:1) coordinate(D);
			\draw[thick] (A1) to[bend left] (-105:1) coordinate(E);
				\foreach \l in {A,A1,B,C,D,E}
			\fill (\l) circle (1.8pt);
			\node at (0,-1.7) {(ii)};
		\end{scope}
		\begin{scope}[xshift=8cm]
			\draw[dotted,thick] (0,0) circle [radius=1];
			\foreach \t/\l in {30/A,-30/A1}{
				\coordinate (\l) at (\t:1);
				\draw[very thick] (\t-10:1) arc (\t-10:\t+10:1);
			}
			\foreach \t/\l in {90/B, -90/C}{
				\coordinate (\l) at (\t:1);
			}
			\draw[very thick] (75:1) arc (75:125:1);
			\draw[very thick] (-75:1) arc (-75:-125:1);
			\draw[thick] (B) -- (C);
			\node[left] at ($(B)!0.5!(C)$) {\footnotesize $a$};
			\draw[thick] (A) -- (-110:1) coordinate(D);
			\draw[thick] (A1) -- (110:1) coordinate(E);
			\foreach \l in {A,A1,B,C,D,E}
				\fill (\l) circle (1.8pt);
			\node at (0,-1.7) {(iii)};
		\end{scope}
		\begin{scope}[xshift=12cm]
			\draw[dotted,thick] (0,0) circle [radius=1];
			\foreach \t/\l in {150/A,-150/A1}{
				\coordinate (\l) at (\t:1);
				\draw[very thick] (\t-10:1) arc (\t-10:\t+10:1);
			}
			\foreach \t/\l in {90/B, -90/C}{
				\coordinate (\l) at (\t:1);
			}
			\draw[very thick] (75:1) arc (75:120:1);
			\draw[very thick] (-75:1) arc (-75:-120:1);
			\draw[thick] (B) -- (C);
			\node[right] at ($(B)!0.5!(C)$) {\footnotesize $a$};
			\draw[thick] (A) -- (-105:1) coordinate(D);
			\draw[thick] (A1) -- (105:1) coordinate(E);
			\foreach \l in {A,A1,B,C,D,E}
			\fill (\l) circle (1.8pt);
			\node at (0,-1.7) {(iv)};
		\end{scope}
	\end{tikzpicture}
\end{center}
\begin{center}
	\begin{tikzpicture}[scale=.8]
		\draw[dotted,thick] (0,0) circle [radius=1];
		\foreach \t/\l in {10/A,190/A1}{
			\coordinate (\l) at (\t:1);
			\draw[very thick] (\t-10:1) arc (\t-10:\t+10:1);
		}
		\foreach \t/\l in {90/B, -90/C}{
			\coordinate (\l) at (\t:1);
		}
		\draw[very thick] (75:1) arc (75:125:1);
		\draw[very thick] (-75:1) arc (-75:-125:1);
		\draw[thick] (B) -- (C);
		\node[right] at ($(B)!0.3!(C)$) {\footnotesize $a$};
		\draw[thick] (110:1) coordinate(D) to[bend left] (A1);
		\draw[thick] (-110:1) coordinate(E) to[bend left] (A);
		\foreach \l in {A,A1,B,C,D,E}
		\fill (\l) circle (1.8pt);
		\node at (0,-1.7) {(v)};
		\begin{scope}[xshift=4cm]
			\draw[dotted,thick] (0,0) circle [radius=1];
			\foreach \t/\l in {10/A,190/A1}{
				\coordinate (\l) at (\t:1);
				\draw[very thick] (\t-10:1) arc (\t-10:\t+10:1);
			}
			\foreach \t/\l in {90/B, -90/C}{
				\coordinate (\l) at (\t:1);
			}
			\draw[very thick] (75:1) arc (75:125:1);
			\draw[very thick] (-75:1) arc (-75:-125:1);
			\draw[thick] (B) -- (C);
			\node[left] at ($(B)!0.5!(C)$) {\footnotesize $a$};
			\draw[thick] (110:1) coordinate(D) to[bend right] (A);
			\draw[thick] (-110:1) coordinate(E) to[bend right] (A1);
			\foreach \l in {A,A1,B,C,D,E}
				\fill (\l) circle (1.8pt);
			\node at (0,-1.7) {(vi)};
		\end{scope}
		\begin{scope}[xshift=8cm]
			\draw[dotted,thick] (0,0) circle [radius=1];
			\foreach \t/\l in {90/A,270/B}{
				\coordinate (\l) at (\t:1);
			}
			\draw[very thick] (80:1) arc (80:130:1);
			\draw[very thick] (230:1) arc (230:280:1);
			\draw[thick] (A) -- (B);
			\node[right] at ($(A)!0.5!(B)$) {\footnotesize $a$};
			\draw[thick] (120:1) coordinate(C) to[bend left] (240:1) coordinate(D);
			\node at ($(C)!0.5!(D)$) {\footnotesize $\tilde{a}$};
			\foreach \l in {A,B,C,D,E}
				\fill (\l) circle (1.8pt);
			\node at (0,-1.7) {(vii)};
		\end{scope}
	\end{tikzpicture}
\end{center}

The case (vii) is trivial, since the recurrence relation \eqref{RecRel} holds for the chord $\tilde{a}. $ 
The cases (i) to (vi) can be proved similarly, so we consider (i). 
We suppose that there are $k$ chords intersecting the chord $a$ and let $b_1, b_k$ be the chords indicated in the diagram (i) which we denoted by $D$. 
Applying \eqref{ten1} to the diagram to obtain:
\begin{align}
	D = 
	\begin{tikzpicture}[scale=.8,baseline=-3pt]
		\draw[dotted,thick] (0,0) circle [radius=1];
		\foreach \t/\l in {27/A,-27/A1}{
			\coordinate (\l) at (\t:1);
			\draw[very thick] (\t-10:1) arc (\t-10:\t+10:1);
		}
		\foreach \t/\l in {92/B, -92/C}{
			\coordinate (\l) at (\t:1);
		}
		\draw[very thick] (75:1) arc (75:125:1);
		\draw[very thick] (-75:1) arc (-75:-125:1);
		\draw[thick] (B) -- (C);
		\node[left] at ($(B)!0.5!(C)$) {\footnotesize $a$};
		\draw[thick] (110:1) coordinate(D) to[bend right] (A);
		\node[below] at ($(D)!0.7!(A)$) {\footnotesize $b_1$};
		\draw[thick] (-110:1) coordinate(E) to[bend left] (A1);
		\node[above] at ($(E)!0.7!(A1)$) {\footnotesize $b_k$};
		\foreach \l in {A,A1,B,C,D,E}
			\fill (\l) circle (1.8pt);
	\end{tikzpicture}
	= 
	\begin{tikzpicture}[scale=.8,baseline=-3pt]
		\draw[dotted,thick] (0,0) circle [radius=1];
		\foreach \t/\l in {27/A,-27/A1}{
			\coordinate (\l) at (\t:1);
			\draw[very thick] (\t-10:1) arc (\t-10:\t+10:1);
		}
		\foreach \t/\l in {92/B, -92/C}{
			\coordinate (\l) at (\t:1);
		}
		\draw[very thick] (75:1) arc (75:125:1);
		\draw[very thick] (-60:1) arc (-60:-100:1);
		\draw[thick] (B) -- (C);
		\node[left] at ($(B)!0.5!(C)$) {\footnotesize $a$};
		\draw[thick] (110:1) coordinate(D) to[bend right] (A);
		\node[below,yshift=-1pt] at ($(D)!0.7!(A)$) {\footnotesize $b_1$};
		\draw[thick] (-77:1) coordinate(E) to[bend left] (A1);
		\foreach \l in {A,A1,B,C,D,E}
		\fill (\l) circle (1.8pt);
	\end{tikzpicture}
	+
	\begin{tikzpicture}[scale=.8,baseline=-3pt]
		\draw[dotted,thick] (0,0) circle [radius=1];
		\foreach \t/\l in {27/A,-27/A1}{
			\coordinate (\l) at (\t:1);
			\draw[very thick] (\t-10:1) arc (\t-10:\t+10:1);
		}
		\foreach \t/\l in {92/B, -92/C}{
			\coordinate (\l) at (\t:1);
		}
		\draw[very thick] (60:1) arc (60:100:1);
		\draw[very thick] (-75:1) arc (-75:-125:1);
		\draw[thick] (B) -- (C);
		\node[left] at ($(B)!0.5!(C)$) {\footnotesize $a$};
		\draw[thick] (77:1) coordinate(D) to[bend right] (A);
		\draw[thick] (-110:1) coordinate(E) to[bend left] (A1);
		\node[above,yshift=1pt] at ($(E)!0.7!(A1)$) {\footnotesize $b_k$};
		\foreach \l in {A,A1,B,C,D,E}
		\fill (\l) circle (1.8pt);
	\end{tikzpicture}
	-
	\begin{tikzpicture}[scale=.8,baseline=-3pt]
		\draw[dotted,thick] (0,0) circle [radius=1];
		\foreach \t/\l in {27/A,-27/A1}{
			\coordinate (\l) at (\t:1);
			\draw[very thick] (\t-10:1) arc (\t-10:\t+10:1);
		}
		\foreach \t/\l in {92/B, -92/C}{
			\coordinate (\l) at (\t:1);
		}
		\draw[very thick] (60:1) arc (60:100:1);
		\draw[very thick] (-60:1) arc (-60:-100:1);
		\draw[thick] (B) -- (C);
		\node[left] at ($(B)!0.5!(C)$) {\footnotesize $a$};
		\draw[thick] (77:1) coordinate(D) to[bend right] (A);
		\draw[thick] (-77:1) coordinate(E) to[bend left] (A1);
		\foreach \l in {A,A1,B,C,D,E}
		\fill (\l) circle (1.8pt);
	\end{tikzpicture}
	+ \epsilon \Lambda_{a,1k} + \epsilon (c-y) \Gamma_{a,1k} 
	\label{Dno1}
\end{align}
where we introduced $ \Lambda_{a,ij} $ and $ \Gamma_{a,ij}$ which are defined by using the notations of Theorem \ref{THM:recrel}: 
\begin{align}
	\Lambda_{a,ij} &= D^{\parallel}_{a,ij} - D^{\times}_{a,ij},
	\notag \\
	\Gamma_{a,ij} &= D_{a,ij}^{\ell r} + D_{a,ij}^{r \ell} - D_{a,ij}^{\ell \ell} - D_{a,ij}^{rr}. \label{LGdef}
\end{align}
By the induction hypothesis
\begin{align}
	\begin{tikzpicture}[scale=.8,baseline]
		\draw[dotted,thick] (0,0) circle [radius=1];
		\foreach \t/\l in {27/A,-27/A1}{
			\coordinate (\l) at (\t:1);
			\draw[very thick] (\t-10:1) arc (\t-10:\t+10:1);
		}
		\foreach \t/\l in {92/B, -92/C}{
			\coordinate (\l) at (\t:1);
		}
		\draw[very thick] (75:1) arc (75:125:1);
		\draw[very thick] (-60:1) arc (-60:-100:1);
		\draw[thick] (B) -- (C);
		\node[left] at ($(B)!0.5!(C)$) {\footnotesize $a$};
		\draw[thick] (110:1) coordinate(D) to[bend right] (A);
		\node[below,yshift=-1pt] at ($(D)!0.7!(A)$) {\footnotesize $b_1$};
		\draw[thick] (-77:1) coordinate(E) to[bend left] (A1);
		\foreach \l in {A,A1,B,C,D,E}
		\fill (\l) circle (1.8pt);
	\end{tikzpicture}
	&= \big( c-\epsilon (k-1) \big) D_a + \epsilon (c-y) \sum_{1 \leq i < k} D_{a,i} + \epsilon \sum_{1\leq i<j<k} \Lambda_{a,ij} 
	\notag \\
	&+ \epsilon (c-y) \sum_{1\leq i<j<k} \Gamma_{a,ij},
\end{align}
\begin{align}
	\begin{tikzpicture}[scale=.8,baseline]
		\draw[dotted,thick] (0,0) circle [radius=1];
		\foreach \t/\l in {27/A,-27/A1}{
			\coordinate (\l) at (\t:1);
			\draw[very thick] (\t-10:1) arc (\t-10:\t+10:1);
		}
		\foreach \t/\l in {92/B, -92/C}{
			\coordinate (\l) at (\t:1);
		}
		\draw[very thick] (60:1) arc (60:100:1);
		\draw[very thick] (-75:1) arc (-75:-125:1);
		\draw[thick] (B) -- (C);
		\node[left] at ($(B)!0.5!(C)$) {\footnotesize $a$};
		\draw[thick] (77:1) coordinate(D) to[bend right] (A);
		\draw[thick] (-110:1) coordinate(E) to[bend left] (A1);
		\node[above,yshift=1pt] at ($(E)!0.7!(A1)$) {\footnotesize $b_k$};
		\foreach \l in {A,A1,B,C,D,E}
		\fill (\l) circle (1.8pt);
	\end{tikzpicture}
	&= \big( c-\epsilon (k-1) \big) D_a + \epsilon (c-y) \sum_{1 < i \leq k} D_{a,i} + \epsilon \sum_{1 < i<j \leq k} \Lambda_{a,ij} 
	\notag \\
	&+ \epsilon (c-y) \sum_{1 < i<j \leq k} \Gamma_{a,ij},
\end{align}
\begin{align}
	\begin{tikzpicture}[scale=.8,baseline]
		\draw[dotted,thick] (0,0) circle [radius=1];
		\foreach \t/\l in {27/A,-27/A1}{
			\coordinate (\l) at (\t:1);
			\draw[very thick] (\t-10:1) arc (\t-10:\t+10:1);
		}
		\foreach \t/\l in {92/B, -92/C}{
			\coordinate (\l) at (\t:1);
		}
		\draw[very thick] (60:1) arc (60:100:1);
		\draw[very thick] (-60:1) arc (-60:-100:1);
		\draw[thick] (B) -- (C);
		\node[left] at ($(B)!0.5!(C)$) {\footnotesize $a$};
		\draw[thick] (77:1) coordinate(D) to[bend right] (A);
		\draw[thick] (-77:1) coordinate(E) to[bend left] (A1);
		\foreach \l in {A,A1,B,C,D,E}
		\fill (\l) circle (1.8pt);
	\end{tikzpicture}
	&= \big( c-\epsilon (k-1) \big) D_a + \epsilon (c-y) \sum_{1 < i < k} D_{a,i} + \epsilon \sum_{1 < i<j < k} \Lambda_{a,ij} 
	\notag \\
	&+ \epsilon (c-y) \sum_{1 < i<j < k} \Gamma_{a,ij}
\end{align}
Putting these three into \eqref{Dno1}, one may see that the recurrence relation \eqref{RecRel} holds for $D$. 

To complete the proof, we need to show the cases of $ \ell = 0 $ and $ \ell = 1.$ 
If $ \ell = 0, $ the diagram $D$ has the shape of the LHS of \eqref{isofac} and the relation \eqref{isofac} is identical to the recurrence relation. 
If $ \ell = 1, $ then there is only one chord intersecting $a$. 
We apply the STU relation and Proposition \ref{PROP:Y} to obtain:
\begin{align}
	\begin{tikzpicture}[scale=.8,baseline=-3pt]
		\draw[dotted,thick] (0,0) circle [radius=1];
		\draw[very thick] (120:1) arc (120:240:1);
		\draw[very thick] (-10:1) arc (-10:10:1);
		\draw[thick] (0:1) coordinate(A) -- (180:1) coordinate(C);
		\draw[thick] (130:1) coordinate(B) to[bend left] (230:1) coordinate(D);
		\foreach \l in {A,B,C,D}
			\fill (\l) circle (1.8pt);
		\node[right,xshift=2pt] at ($(B)!0.25!(D)$) {\footnotesize $a$};
		\node[above] at ($(A)!0.3!(C)$) {\footnotesize $b_1$};
	\end{tikzpicture}
	&=
	\begin{tikzpicture}[scale=.8,baseline=-3pt]
		\draw[dotted,thick] (0,0) circle [radius=1];
		\draw[very thick] (120:1) arc (120:240:1);
		\draw[very thick] (-10:1) arc (-10:10:1);
		\draw[thick] (0:1) coordinate(A) to[bend left] (130:1) coordinate(C);
		\draw[thick] (150:1) coordinate(B) to[bend left] (230:1) coordinate(D);
		\foreach \l in {A,B,C,D}
			\fill (\l) circle (1.8pt);
	\end{tikzpicture}
	-
	\begin{tikzpicture}[scale=.8,baseline=-3pt]
		\draw[dotted,thick] (0,0) circle [radius=1];
		\draw[very thick] (120:1) arc (120:240:1);
		\draw[very thick] (-10:1) arc (-10:10:1);
		\draw[thick] (0:1) coordinate(A) -- (180:0.5) coordinate(v);
		\draw[thick] (140:1) coordinate(B) -- (v) -- (220:1) coordinate(D);
		\foreach \l in {A,B,D}
			\fill (\l) circle (1.8pt);
	\end{tikzpicture}
	\notag \\
	&= c 
	\begin{tikzpicture}[scale=.8,baseline=-3pt]
		\draw[dotted,thick] (0,0) circle [radius=1];
		\draw[very thick] (120:1) arc (120:240:1);
		\draw[very thick] (-10:1) arc (-10:10:1);
		\draw[thick] (0:1) coordinate(A) -- (180:1) coordinate(B);
		\foreach \l in {A,B}
			\fill (\l) circle (1.8pt);
	\end{tikzpicture}
	-\epsilon 
	\left(
		\begin{tikzpicture}[scale=.8,baseline=-3pt]
			\draw[dotted,thick] (0,0) circle [radius=1];
			\draw[very thick] (120:1) arc (120:240:1);
			\draw[very thick] (-10:1) arc (-10:10:1);
			\draw[thick] (0:1) coordinate(A) -- (180:1) coordinate(B);
			\foreach \l in {A,B}
			\fill (\l) circle (1.8pt);
		\end{tikzpicture}
		- (c-y)
		\begin{tikzpicture}[scale=.8,baseline=-3pt]
			\draw[dotted,thick] (0,0) circle [radius=1];
			\draw[very thick] (120:1) arc (120:240:1);
			\draw[very thick] (-10:1) arc (-10:10:1);
		\end{tikzpicture}
	\right)
	\notag\\
	&= (c-\epsilon) D_a + \epsilon (c-y) D_{a,1}.
\end{align} 
This is identical to the recurrence relation \eqref{RecRel}. 

\section{Deframed $A1_{\epsilon}$ universal weight system} \label{SEC:deframing}
\setcounter{equation}{0}

So far we have discussed the universal weight system $w$ for framed knots. 
In this section, we consider a universal weight system $\bar{w}$ of ordinary (unframed) knots and show the following relation:
\begin{thm} \label{THM:wbarw}
	For any chord diagram $ D \in \mathcal{A}$
	\begin{equation}
		\bar{w}(h) = w(c=0,h). \label{wbarw}
	\end{equation}
\end{thm}
The same relation holds for $gl(1|1)$ weight system \cite{FOKV} and, indeed, one may prove \eqref{wbarw} in a similar way to the case of $gl(1|1)$. 
Namely, by using the deframing operator and Lemma \ref{LEM:sder} below.

Let us recall the deframing operator introduced in \cite{KrSpAi}. 
First, two maps needed to define the deframing operator.
\begin{DEF}
	The linear map $ s : \mathcal{A}_n \to \mathcal{A}_{n-1}$ is defined by
	\begin{equation}
		s(D) = \sum_a D_a
	\end{equation}
	where the sum runs over all chords $a$ in $D$ and $ D_a = D - a. $ If $D \in \mathcal{A}_0$, i.e., $D$ has no chord, then $ s(D) = 0.$ 
	
	The linear map $ \theta : \mathcal{A}_n \to \mathcal{A}_{n+1} $ takes connected sum of $ D \in \mathcal{A}_n$ and the chord diagram with a single chord. 
\end{DEF}
It is easy to see that the maps $ s , \theta $ preserve the 4T relation and $s$ satisfies the Leibniz rule $s(D\cdot D') = s(D) \cdot D' + D\cdot s(D'). $ 

\begin{prop}[\cite{KrSpAi}]
	Define the map $ \phi : \mathcal{A}_n \to \mathcal{A}_n$ by
	\begin{equation}
		\phi = \sum_{k=1}^n \frac{(-1)^k}{k!} \theta^k\cdot s^k. \label{phidef}
	\end{equation}
	Then, we have
	\begin{enumerate}
		\item $ s\circ \phi = 0 $
		\item $ \phi^2 = \phi $
		\item $ \phi $ is the deframing operator:
		\begin{equation}
			\phi\left(
			 	\begin{tikzpicture}[baseline,scale=.8]
			 		\draw[dotted,thick] (0,0) circle [radius =1];
			 		\draw[very thick] (120:1) arc (120:240:1);
			 		\draw[thick] (135:1) coordinate(A) to[bend left] (225:1) coordinate(B);
			 		\foreach \l in {A,B}
			 			\fill (\l) circle (1.8pt);
			 	\end{tikzpicture}
			\right) = 0.
		\end{equation}
	\end{enumerate}
\end{prop}
It follows immediately that
\begin{equation}
	\bar{w}(D) = w (\phi(D)), \quad D \in \mathcal{A}_n \label{wbarphi}
\end{equation}
To calculate the RHS, the following lemma, the same also holds for $gl(1|1)$ \cite{FOKV}, is crucial:
\begin{lemma} \label{LEM:sder}
	For any chord diagram $ D \in \mathcal{A}_n$
	\begin{equation}
		\frac{\partial}{\partial c} w(D) = w(s(D)).  \label{sder}
	\end{equation}
\end{lemma}
\begin{proof}
	We prove by induction on the order $n$ of the diagram (number of chords). If $n=1$, \eqref{sder} is verified immediately.  For a digram $ D \in \mathcal{A}_n,$ we use the recurrence relation \eqref{RecRel} of Theorem \ref{THM:recrel}
	\begin{align}
		\frac{\partial}{\partial c} w(D) &= w(D_a) + (c-\epsilon k) \frac{\partial}{\partial c} w(D_a) + \epsilon \sum_i w(D_{a,i}) + \epsilon (c-y) \sum_i \frac{\partial}{\partial c} w(D_{a,i})
		\notag \\
		&+ \epsilon \sum_{i<j} \frac{\partial}{\partial c} w(\Gamma_{a,ij}) + \epsilon \sum_{i<j} w(\Lambda_{a,ij}) + \epsilon (c-y) \sum_{i<j} \frac{\partial}{\partial c} w(\Lambda_{a,ij})
		\notag \\
		&= w(D_a) + (c-\epsilon k)  w(s(D_a)) + \epsilon \sum_i w(D_{a,i}) + \epsilon (c-y) \sum_i  w(s(D_{a,i}))
		\notag \\
		&+ \epsilon \sum_{i<j}  w(s(\Gamma_{a,ij})) + \epsilon \sum_{i<j} w(\Lambda_{a,ij}) + \epsilon (c-y) \sum_{i<j}  w(s(\Lambda_{a,ij}))
		\label{delwd}
	\end{align}
	where $ \Gamma_{a,ij}, \Lambda_{a,ij}$ are defined in \eqref{LGdef} and the second equality is due to the induction hypothesis. 
	
	We denote the chords intersecting $a$ by latin letters, e.g., $ b_i, b_j, $ and those not intersecting $a$ by greek letters, e.g., $ b_{\mu}, b_{\nu}.$ 
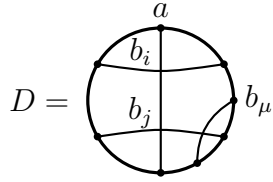
	\begin{center}
	\begin{tikzpicture}[scale=0.8]
		\draw[very thick] (0,0) circle [radius=1.2];
		\foreach \t/\label in {30/A,90/B,150/C,210/D,270/E,330/F}{
			\coordinate (\label) at ({1.2*cos(\t)},{1.2*sin(\t)});
			\fill (\label) circle (1.8pt);
		}
		\draw[thick] (B) node[above] {$a$}--(E);
		\draw[thick,] (A) ..controls ($(A)!0.5!(C)+(0,-0.15)$).. (C) node[pos=0.7,yshift=7pt] {$b_i$};
		\draw[thick,bend left=10] (D)..controls ($(D)!0.5!(F)+(0,0.15)$).. (F) node[pos=0.3,yshift=7pt] {$b_j$};
		\draw[thick] (0:1.2) coordinate (G) to[bend right] (-60:1.2) coordinate(H);
		\foreach \l in {G,H}
			\fill (\l) circle (1.8pt);
		\node[right] at (G)  {$b_{\mu}$}; 
		\node at (-2,0) {$D = $};
	\end{tikzpicture}
	\end{center}
	Then, we have
	\begin{align}
		s(D_a) &= \sum_i D_{a,i} + \sum_{\mu} D_{a,\mu},
		\notag \\
		s(D_{a,i}) &= \sum_{j} D_{a,ij} + \sum_{\mu} D_{a,i\mu},
		\notag \\
		s(\Gamma_{a,ij}) &= \sum_{\ell} \Gamma_{a,ij\ell} + \sum_{\mu} \Gamma_{a,ij\mu} - \Lambda_{a,ij},
		\notag \\
		s(\Lambda_{a,ij}) &= \sum_{\ell} \Lambda_{a,ij\ell} + \sum_{\mu} \Lambda_{a,ij\mu}. 
		\label{sDs}
	\end{align}
	
	On the other hand, we have
	\begin{equation}
		w(s(D)) = w(D_a) + \sum_i w(D_i) + \sum_{\mu} w(D_{\mu}). \label{wsd}
	\end{equation}
	Using the recurrence relation \eqref{RecRel}, we get the followings:
	\begin{align}
		w(D_i) &= \big( c-\epsilon (k-1) \big) w(D_{a,i}) + \epsilon (c-y) \sum_j w(D_{a,ij}) + \epsilon \sum_{j < \ell} w(\Gamma_{a,ij\ell}) 
		\notag \\
		&+ \epsilon (c-y) \sum_{j < \ell} w(\Lambda_{a,ij\ell}),
		\notag \\
		w(D_{\mu}) &= ( c-\epsilon k) w(D_{a,i}) + \epsilon (c-y) \sum_i w(D_{a,i\mu}) + \epsilon \sum_{i<j} w(\Gamma_{a,ij\mu}) 
		\notag \\
		&+ \epsilon (c-y) \sum_{i<j} w(\Lambda_{a,ij\mu}).
	\end{align}
	Substituting these into \eqref{wsd} and using \eqref{sDs}, one may see that $w(s(D))$ is identical to \eqref{delwd}. 
\end{proof} 

\begin{proof}[Proof of Theorem \ref{THM:wbarw}]
	\begin{align}
		\bar{w}(D) &\stackrel{\eqref{wbarphi}}{=} w(\phi(D)) 
		\stackrel{\eqref{phidef}}{=} \sum_{k=0}^n \frac{(-1)^k}{k!} w(\theta^k) w(s^k(D)) 
		\stackrel{\eqref{sder}}{=} \sum_{k=0}^n \frac{(-1)^k}{k!} c^k \frac{\partial^k}{\partial c^k} w(D).
	\end{align}
	It is immediate from this expression that $\bar{w}(D)$ is independent of $c:$
	\begin{equation}
		\frac{\partial }{\partial c} \bar{w}(D) = 0. 
	\end{equation}
	Therefore,
	\begin{align}
		\bar{w}(D) = \bar{w}(D) \Big|_{c=0} =  \left. \sum_{k=0}^n \frac{(-1)^k}{k!} c^k \frac{\partial^k}{\partial c^k} w(D)\right|_{c=0} =  w(D) \Big|_{c=0}.
	\end{align}
\end{proof}

An immediate corollary of Theorem \ref{THM:wbarw} is the recurrence relation for the weight systems of ordinary knots:
\begin{cor}
	\begin{align}
		\bar{w}(D) &= - \epsilon \left(
		  k \bar{w}(D_a) + y \sum_i \bar{w}(D_{a,i}) - \sum_{i<j} \bar{w}(\Gamma_{a,ij}) + y \sum_{i<j} \bar{w}(\Lambda_{a,ij}).
		\right)
	\end{align}
\end{cor}

\section{Conclusions} \label{SEC:CR}
\setcounter{equation}{0}

We showed that color Lie algebras are a class of $S$-Lie algebras. 
This allows us to construct universal weight systems from the color Lie algebras. As an example of color Lie algebras, we took one of the minimal $\Z2$-graded Lie algebras $A1_{\epsilon}$ and studied the weight system coming from it. 
It is a polynomial in two-variables $c, h$ and the specialization of $c = 0 $ gives the weight system of unframed knots. We derived some relations such as \eqref{RecRel} for the polynomials in a similar way to $sl(2)$ and $gl(1|1)$. It then turned out that the $A1_{\epsilon}$ weight system is a hybrid of those from $sl(2) $ and $gl(1|1)$.    
This hybrid property is a reflection of the algebraic structure of $A1_{\epsilon}$, it has a one-dimensional center spanned by $H$ and $ Q_i, (i=1, 2, 3)$ are realized by $ so(3) \simeq sl(2) $ and the quaternion:
\begin{align}
	Q_i = e_i \otimes L_i, 
\end{align}
where $ e_i,$ and $L_i$ satisfy the relations
\begin{equation}
	e_i e_j = -\delta_{ij} + i \sum_k \epsilon_{ijk} e_k, \quad 
	[L_i, L_j] = \sum_k \epsilon_{ijk} L_k
\end{equation}
with $ \epsilon_{ijk} $ is the totally antisymmetric and $ \epsilon_{123} = 1. $ 

This implies that if we take an another $\Z2$-graded Lie (super)algebra, the resulting weight system will have different properties. 
An interesting example from the minimal $\Z2$-graded superalgebra is the $\Z2$-graded super-Poincar\'e algebra in $(0+1)$-dimension \cite{Bruce}. 
Most of the minimal $\Z2$-graded Lie (super)algebras are almost abelian, i.e., many (anti)commutators vanish, so we should go beyond the minimal algebras. 
The $\Z2$-graded version of $sl(2)$ and $osp(1|2)$ discussed in \cite{niktt,AizawaSegar} provides  an example of beyond the minimal ones. In particular, they have the $(1,1)$-graded Casimir $c_{11}$ in addition to the ordinary Casimir $c_{00}$ of $(0,0)$-grading. Therefore, the weight systems constructed from these will be polynomials of two variables $ c_{00} $ and $c_{11}^2$ as $ \deg c^2_{11} = (0,0). $

Another challenging problem is the use of color Lie algebras with more complex gradings such as $ \mathbb{Z}_3, \mathbb{Z}_2 \times \mathbb{Z}_4.$ 
The defining relations of these algebras are no longer (anti)commutators. 
Therefore, a precise analysis of structure and representations is required before investigating weight systems. 

The present work has established a connection between knot invariants and color Lie algebras. The connection will not be limited to the weight systems discussed here. Further studies are needed for a deeper understanding of the role of color Lie algebras in knot theory and related studies of three-dimensional manifolds.

\section*{Acknowledgments}

N. A. is supported by JSPS KAKENHI Grant Number JP23K03217. 

\section*{Data availability}
No datasets were generated or analysed during the current study

\section*{Conflict of interest}

All authors have no conflicts of interest.

\appendix
\section{Invariant bilinear forms on $\Z2$-graded Lie algebras}

Let $ V = V_{(0,0)} \oplus V_{(1,0)} \oplus V_{(0,1)} \oplus V_{(1,1)} $ be a $\Z2$-graded vector space. We order the basis of $V$ according to the degree:
\begin{equation}
	(0,0), \quad (1,0), \quad (0,1), \quad (1,1). 
\end{equation}
A linear map $ f : V \to V $ of the degree $(0,0)$ may be represented by a block matrix in which each block has a definite degree:
\begin{equation}
	 \begin{pmatrix}
		A^{(0,0)} & A^{(1,0)} & A^{(0,1)} & A^{(1,1)}
		\\
		B^{(1,0)} & B^{(0,0)} & B^{(1,1)} & B^{(0,1)} 
		\\
		C^{(0,1)} & C^{(1,1)} & C^{(0,0)} & C^{(1,0)} 
		\\
		D^{(1,1)} & D^{(0,1)} & D^{(1,0)} & D^{(0,0)} 
	\end{pmatrix}. \label{gradedmatrix}
\end{equation}
We denote the space of such linear maps by $\Z2$-$gl(V)$. 

Now let $ \g = \g_{(0,0)} \oplus \g_{(1,0)} \oplus \g_{(0,1)}  \oplus \g_{(1,1)} $ be a $\Z2$-graded Lie algebra and $ X_1, X_2, \dots $ be a basis fo $\g.$ 
We denote $ \deg(X_a) = \bm{a} $ (this abuse of notation will not cause serious confusion).  
A representation of $\g$ on $V$ is an $\Z2$-graded algebra homomorphism $ \rho : \g \to \Z2$-$gl(V)$. 
Explicitly, $\rho $ should be a linear map and it should satisfies
\begin{align}
  &	\deg(\rho(X_a)) = \deg(X_a),
	\notag \\
  &	\rho(\llbracket X_a, X_b \rrbracket ) = \llbracket \rho(X_a), \rho(X_b) \rrbracket := 
\rho(X_a) \rho(X_b) - \ph{a}{b} \rho(X_b) \rho(X_a).
\end{align}
This means that if, say $\deg(X_a) = (1,0)$, the representation matrix $\rho(X_a)$ has non-vanishing entries only at the $(1,0)$ blocks in \eqref{gradedmatrix}. 
For the $\Z2$-graded Lie algebras, the trace operation of the representation matrix is defined by 
\cite{rw2}
\begin{equation}
	tr \rho(X_a) = tr A^{(0,0)} + tr B^{(0,0)} + tr C^{(0,0)} + tr D^{(0,0)}.
\end{equation}
It follows from this definition that the trace of the $\Z2$-graded matrix has the following properties
\begin{align}
	& tr F = 0 \quad \text{for} \ \deg F \neq (0,0),
	\notag \\
	& tr(FG) = tr(GF). 
\end{align}

With these preparations, we now define bilinear forms on $\g$. 
Let $ \rho $ be an $n$-dimensional representation of $\g$ (not necessarily irreducible) and $ M $ be an $ n \times n $ $\Z2$-graded matrix of degree $\bm{m}$ which satisfies
\begin{equation}
	\llbracket M, \rho(X_a) \rrbracket = 0, \quad \forall X_a \in \g \label{MatrixM}
\end{equation}
\begin{DEF} \label{DEF:biform}
	A degree $\bm{m} $ bilinear form $ B : \g \times \g \to \mathbb{C} $ is defined by
	\begin{equation}
		B(X_a,X_b) = tr(\rho(X_a) M \rho(X_b)), \quad X_a, X_b \in \g 
	\end{equation}
\end{DEF}
\begin{prop} \label{PROP:propertybi}
	The bilinear form $B$ satisfies the followings:
	\begin{enumerate}
		\renewcommand{\labelenumi}{(\roman{enumi})}
		\item $ B(X_a,X_b) = 0 $ if $ \bm{a} + \bm{b} \neq \bm{m}$.
		\item $ B(X_a,X_b) = \ph{a}{m} B(X_b,X_a) = \ph{b}{m} B(X_b,X_a)$.
		\item $ B(\llbracket X_a, X_b \rrbracket, X_c) = \ph{b}{m} B(X_a,\llbracket X_b, X_c \rrbracket)$. 
	\end{enumerate}
	The last one is the adjoint invariance of $B$, so the bilinear form of Definition \ref{DEF:biform} is an invariant bilinear form. 
\end{prop}
\begin{proof}	
\begin{enumerate}
	\renewcommand{\labelenumi}{(\roman{enumi})}
	\item Obvious from the definition of the trace operation. 
	\item We write $ A:= \rho(X_a), B := \rho(X_b)$ for better readability. Using the cyclic property of the trace
	   \begin{align*}
	   	B(X_a,X_b) &= tr(A M B) = tr(B A M) 
	   	\stackrel{\eqref{MatrixM}}{=} \ph{a}{m} tr(B M A) = \ph{a}{m} B(X_b,X_a).
	   \end{align*}
	   By the property (i), it is sufficient to consider the case of $ \bm{a} + \bm{b} = \bm{m}.$ 
	   For this case, $ \DP{a}{m} = (\bm{b} + \bm{m}) \cdot \bm{m} = \DP{b}{m} \mod (2,2)$.
	 \item  Note taht the bilinear form on the RHS is
	   \begin{align}
	   	B(X_a, \llbracket X_b, X_c \rrbracket) &= tr(A M \rho(\llbracket X_b,X_c \rrbracket))
	   	= tr(A M B C) -\ph{b}{c} tr(A M C B).
	   \end{align} 
	   Similarly, that on the LHS is 
	   \begin{align}
	   	 B(\llbracket X_a, X_b \rrbracket, X_c) 
	   	 &= tr(AB M C) - \ph{a}{b} tr(B A M C) 
	   	 \notag \\
	   	 &\stackrel{\eqref{MatrixM}}{=} \ph{b}{m} tr(A M B C)- \ph{a}{b} tr(A M C B).
	   	 \notag 
	   \end{align}
	   Noting that $ \DP{a}{b} = \bm{b}\cdot (\bm{a}+\bm{c}) + \DP{b}{c} = \bm{b} \cdot (\bm{b}+\bm{m}) + \DP{b}{c} = \DP{b}{m} + \DP{b}{c}$, one may immediately see that (iii) holds. 
\end{enumerate}
\end{proof}

 Let us now write down the invariance of $B$ in terms of the structure constants of $\g$: 
\begin{equation}
	\llbracket X_a, X_b \rrbracket = \sum_c \f{ab}{c} X_c.
\end{equation}
Introducing the notation $ B_{ab} := B(X_a, X_b)$, the invariance (iii) is written as
\begin{equation}
	\sum_d \f{ab}{d} B_{dc} = \ph{b}{m} \sum_d \f{bc}{d} B_{ad}. \label{InvB}
\end{equation}
In the case of the bilinear form $B$ is non-degenerate, we denote its inverse by $C: $ 
$ \sum_c B_{ac} C^{cb} = \delta_a^b. $ 
Note that the forms $ B $ and $ C$ have the same $\Z2$-degree. 
Then, \eqref{InvB} is equivalent to
\begin{equation}
	\sum_d C^{ad} \f{db}{c} = \ph{b}{m} \sum_d \f{bd}{a} C^{dc}. \label{InvC}
\end{equation}

The invariant bilinear form $C$ is used to introduce the $ \Z2$-graded Casimir elements. 
\begin{prop} \label{PROP:CasimirDef}
	\begin{equation}
		C^{(\bm{m})} := \sum_{ab} C^{ab} X_a X_b
	\end{equation}
	is the $\Z2$-graded Casimir element of degree $\bm{m}:$ 
	\begin{equation}
		\llbracket X_a, C^{(\bm{m})} \rrbracket = 0, \quad \forall X_a \in \g \label{CasimirRel}
	\end{equation}
\end{prop}
\begin{proof} 
	\eqref{CasimirRel} can be proved by direct computation as follows:
	\begin{align*}
		\llbracket X_a, C^{(\bm{m})} \rrbracket &= \sum_{b,c} C^{bc}  \llbracket X_a,  X_b X_c \rrbracket 
		= \sum_{bc} C^{bc} ( \llbracket X_a, X_b \rrbracket X_c + \ph{a}{b} X_b \llbracket X_a, X_c \rrbracket )
		\notag \\
		&= \sum_{b,c,d} ( C^{cd} \f{ac}{b} + \ph{a}{b} \f{ac}{d} C^{bc} ) X_b X_d
		\notag \\
		&= -\ph{a}{m} \sum_{b,c,d} ( C^{bc} \f{ca}{d} - \ph{a}{m} \f{ac}{b} C^{cd} ) X_b X_d 
		\stackrel{\eqref{InvC}}{=} 0.
	\end{align*}
\end{proof}

The bilinear forms on $\Z2$-graded Lie superalgebras have been discussed in \cite{AizawaSegar} and properties similar to Proposition \ref{PROP:propertybi} have been observed. 

%
%

\end{document}